\documentclass[a4paper,12pt]{article}
\usepackage[T1]{fontenc}
\usepackage{amsfonts,amsthm,amssymb,mathrsfs,amsmath}
\usepackage{graphicx,graphics}
\usepackage{relsize}
\usepackage[all]{xy}
\usepackage{tikz}
\allowdisplaybreaks

\newcommand{\ZZ}{\mathbb{Z}}
\newcommand{\NN}{\mathbb{N}}
\newcommand{\Nc}{\mathcal{N}}
\newcommand{\CC}{\mathbb{C}}
\newcommand{\RR}{\mathbb{R}}
\newcommand{\Hi}{\mathscr{H}}

\newcommand{\dprod}[2]{\left\langle #1,#2\right\rangle}
\newcommand{\norm}[1]{\left\lVert #1\right\rVert}
\newcommand{\proj}[1]{\left|#1\right\rangle\left\langle#1\right|}

\newcommand{\matx}[1]{\left(\begin{matrix} #1 \end{matrix}\right)}
\newcommand{\clsp}[1]{\overline{\left\langle #1 \right\rangle}}
\newcommand{\essSup}[1]{\mathop{\text{ess-sup}}_{#1}}

\newtheorem{theorem}{Theorem}[section]
\newtheorem{lemma}[theorem]{Lemma}

\newtheorem{remark}[theorem]{Remark}
\numberwithin{equation}{section}

\newtheorem{corollary}[theorem]{Corollary}

\title{Multiplicity theorem of singular Spectrum for general Anderson type Hamiltonian}
\author{Anish Mallick\footnote{E-mail:\texttt{anish.mallick@icts.res.in}, Institute: ICTS-TIFR Bangaluru, India.} \& Dhriti Ranjan Dolai\footnote{E-mail:\texttt{dhriti\_vs@isibang.ac.in}, Institute: ISI Bangalore, India.}}
\date{\today}

\begin{document}
\maketitle
\begin{abstract}
In this work, we focus on the multiplicity of singular spectrum for operators of the form $A^\omega=A+\sum_{n}\omega_n C_n$ on a separable Hilbert space $\Hi$, 
where $A$ is a self adjoint operator and $\{C_n\}_{n}$ is a countable collection of non-negative finite rank operators.
When $\{\omega_n\}_n$ are independent real random variables with absolutely continuous distributions,
we show that the multiplicity of the singular spectrum is almost surely bounded above by the maximum algebraic multiplicity of eigenvalues of 
$\sqrt{C_n}(A^\omega-z)^{-1}\sqrt{C_n}$ for all $n$ and almost all $(z,\omega)$.
The result is optimal in the sense that there are operators for which the bound is achieved. 
Using this, we also provide an effective bound on the multiplicity of singular spectrum for some special cases.
\end{abstract}
\section{Introduction}
Spectral theory of random operators is an important field of study, and within it, the Anderson tight binding model and random Schr\"{o}dinger operator have gained significant attention.
Over the years much attention has been given to the nature of their spectrum. 
But to completely characterize the structure of the operator, information on the multiplicity is also important.
Here we pay attention to the multiplicity of the singular spectrum for certain class of random operators.

One of the well studied class of random operators is the Anderson tight binding model.
Many results about its spectrum are known, for example: the existence of pure point spectrum is known for Anderson tight binding model over integer lattice \cite{AM,CKM,FMSS,JS}.
Absolutely continuous spectrum is known to exist for Anderson tight binding model over Bethe lattice \cite{FHS,K1}.
Other models where the pure point spectrum is known to exist includes random Schr\"{o}dinger operator \cite{B1,CH1,GK2,K3}, 
multi-particle Anderson model \cite{AW1,CMS1,KN} and magnetic Schr\"{o}dinger operators \cite{CH2,W1}.

There are important results which also concentrate on the multiplicity of the singular spectrum.
For the Anderson tight binding model, Simon \cite{BS2}, Klein-Molchanov \cite{KM2} have shown the simplicity of pure point spectrum.
For Anderson type models when the randomness acts as rank one perturbations, Jak{\v{s}}i{\'c}-Last \cite{JL1,JL2} showed that the singular spectrum is simple.
For random Schr\"{o}dinger operator, in the regime of exponential decay of Green's function, Combes-Germinet-Klein \cite{CGK} showed that the spectrum is simple. 
Other work includes \cite{SSH}, where Sadel and Schulz-Baldes provided multiplicity result for absolute continuous spectrum for random Dirac operators with time reversal symmetry.
But general results concerning the multiplicity of the spectrum are not known.
One of the difficulties involving multiplicity results for random Schr\"{o}dinger operator or multi-particle Anderson model is that the randomness acts as perturbation over an infinite rank operator.

Randomness acting through perturbation by a finite rank operator is an intermediate step between Anderson tight binding model and random Schr\"{o}dinger operator. 
Some example of such a random operator is Anderson dimer/polymer model, Toeplitz/Hankel random matrix, and random conductance model.
Here we will deal with Anderson type operators and provide multiplicity result for the singular spectrum when the randomness acts through perturbation by a finite rank non-negative operator.
This work is similar to the work done by Jak{\v{s}}i{\'c}-Last \cite{JL1,JL2} and is a generalization and extension of the work done by Mallick \cite{AM1}. 
Though this work does not answer the question about the multiplicity of singular spectrum for random Schr\"{o}dinger operator but it is a step towards it.
The technique involved in the proof does not distinguish between point spectrum and singular continuous spectrum, so stated results are true for whole of the singular spectrum.

For a densely defined self-adjoint operator $A$ with domain $\mathcal{D}(A)$ on a separable Hilbert space $\Hi$ and 
a countable collection of finite rank non-negative operator $\{C_n\}_{n\in\Nc}$, define the random operator 
\begin{equation}\label{eqn::mainOpEq1}
A^\omega=A+\sum_{n\in \Nc} \omega_n C_n,
\end{equation}
where $\{\omega_n\}_{n\in\Nc}$ are independent real random variables with absolutely continuous distribution.
Let $(\Omega,\mathcal{B},\mathbb{P})$ denote the probability space such that $\omega_n$ are random variables over $\Omega$. 
We will assume that 		
$$A^\cdot:\Omega\rightarrow\mathcal{S}(\Hi)$$
is an essentially self adjoint operator valued random variable.
This is a necessary assumption because otherwise there can be multiple self-adjoint extensions for the symmetric operator $A^\omega$. 
The assumption itself is not too restrictive and a large class of operators satisfy this condition. 
For example, if $A$ is bounded self-adjoint, $\{C_n\}_n$ are finite rank non-negative operators satisfying $C_nC_m=C_mC_n=0$ for any $n\neq m$ and
the distributions of the random variables $\omega_n$ are supported in some fixed compact set $[-K,K]$, then the operator $A^\omega$ is almost surely bounded and self adjoint. 
Anderson polymer/dimer model falls into this category of operators.

For the main result we need to focus on the linear maps
$$G^\omega_{n,n}(z):=P_n(A^\omega-z)^{-1}P_n:P_n\Hi\rightarrow P_n\Hi$$
for $z\in\CC\setminus\RR$, where $P_n$ is projection onto the range of $C_n$. 
Using functional calculus, it is easy to see that the linear operator $G^\omega_{n,n}(z)$ can be viewed as a matrix over $P_n\Hi$ (after fixing a basis of $P_n\Hi$), 
which belongs to the set of matrix valued Herglotz functions.
Using the representation of matrix-valued Herglotz functions (see \cite[Theorem 5.4]{GT1}), 
we can extract all the properties of the spectral measure over the minimal closed $A^\omega$-invariant subspace containing $P_n\Hi$.

We will use the notation $Mult^\omega_n(z)$ to denote the maximum multiplicity of the root of the polynomial
$$\det(C_n G^\omega_{n,n}(z)-xI)$$
in $x$, for $z\in \CC\setminus\RR$, 
where $C_n$ and $G^\omega_{n,n}(z)$ are viewed as a linear operator on $P_n\Hi$ and so $I$ denotes the identity operator on $P_n\Hi$.
Since $C_n>0$ on $P_n\Hi$ we have
\begin{equation*}
\det(C_n G^\omega_{n,n}(z)-xI)=\det(\sqrt{C_n} G^\omega_{n,n}(z)\sqrt{C_n}-xI),
\end{equation*}
because similarity transformation preserves determinant. 
This is the reason why algebraic multiplicity of $\sqrt{C_n}(A^\omega-z)^{-1}\sqrt{C_n}$ can also be used instead of $C_n G^\omega_{n,n}(z)$.
With these notations, we state our main result:
\begin{theorem}\label{mainThm}
Let $A$ be a densely defined self-adjoint operator with domain $\mathcal{D}(A)$ on a separable Hilbert space $\Hi$ and $\{C_n\}_{n\in\Nc}$ be a countable collection of finite rank non-negative operator. 
Denote $P_n$ to be the projection onto the range of $C_n$ and let $\sum_n P_n=I$.
Let $\{\omega_n\}_{n\in\Nc}$ be a sequence of independent real random variables on the probability space $(\Omega,\mathcal{B},\mathbb{P})$ with absolutely continuous distribution.
Let $A^\omega$ given by \eqref{eqn::mainOpEq1} be a family of essentially self adjoint operators. 
Then
\begin{enumerate}
\item For any $n\in\Nc$
$$\essSup{z\in\CC\setminus\RR}Mult^\omega_n(z)$$
is constant for almost all $\omega$, which will be denoted by $\mathcal{M}_n$.
\item If $\sup_{n\in\Nc} \mathcal{M}_n<\infty$, then the multiplicity of singular spectrum for $A^\omega$ is upper bounded by $\sup_{n\in\Nc} \mathcal{M}_n$, for almost all $\omega$.
\end{enumerate}
\end{theorem}
\begin{remark}\label{rem1}
There are few observations to be made:
\begin{enumerate}
\item Note that if $range(C_n)\subset\mathcal{D}(A)$ for all $n$, then the subspace
$$\mathcal{D}:=\left\{\sum_{i=1}^N \phi_i:\phi_i\in range(C_{n_i}), n_i\in\Nc~\forall 1\leq i\leq N,\forall N\in\NN\right\},$$
is dense and is the domain of $A^\omega$. 
If either $A$ is bounded or $\sup_n |\omega_n|\norm{C_n}$ is finite, then it is easy to show that $A^\omega$ is essentially self adjoint.
\item Note that although $\{C_n\}_n$ are finite rank operators, there may not be a universal upper bound on their ranks. An easy example of such an operator is
$$H^\omega=\Delta+\sum_{n=0}^\infty \omega_n \chi_{\{x:\norm{x}_\infty=n\}},$$
defined on the Hilbert space $\ell^2(\ZZ^d)$, where $\Delta$ is the discrete Laplacian and $\chi_{\{x:\norm{x}_\infty=n\}}$ is projection onto the subspace $\ell^2(\{x\in\ZZ^d:\norm{x}_\infty=n\})$.
\item We need $\sum_n P_n=I$ so that the subspace $\sum_n \Hi^\omega_{P_n}$ is dense in $\Hi$. Here we denote
$$\Hi^\omega_{P_n}=\clsp{f(A^\omega)\phi:f\in C_c(\RR),\phi\in P_n\Hi},$$
where $\clsp{S}$ denotes the closure of finite linear combination of elements of the set $S$.
Without this condition infinite multiplicity could easily be achieved. For example consider the Hilbert space $\oplus^2\ell^2(\ZZ)$, and define the operator
$$H^\omega=\left(\Delta+\sum_{n\in\ZZ} \omega_n\chi_{\{nN,\cdots,(n+1)N-1\}}\right)\oplus \left(\sum_{n\in\ZZ}x_n\proj{\delta_n}\right)$$
where $\{x_n\}_{n\in\ZZ}$ is a fixed sequence and $\{\omega_n\}_{n\in\ZZ}$ are independent real random variables with absolutely continuous distribution. 
Notice that first operator is Anderson like operator with simple point spectrum but the second operator can have arbitrary multiplicity depending upon the sequence $\{x_n\}_n$. 
\end{enumerate}
\end{remark}
\begin{remark}\label{rem2}
To understand the conclusion of the theorem, consider the following examples:
\begin{enumerate}
 \item On the Hilbert space $\ell^2(\ZZ\times\{0,\cdots,N\})$, consider the operator
 $$H^\omega=\tilde{\Delta}+\sum_{n\in\ZZ}\omega_{n}P_{n},$$
 where 
$$(\tilde{\Delta}u)(x,y)=u(x+1,y)+u(x-1,y)\qquad\forall (x,y)\in\ZZ\times\{0,\cdots,N\}$$
and the sequence of projections $P_{n}$ is given by
$$(P_{n}u)(x,y)=\left\{\begin{matrix} u(x,y) & x=n\\ 0 & x\neq n
\end{matrix}\right..$$

\begin{figure}[ht]
\begin{center}
\begin{tikzpicture}[scale=0.7]
\foreach \x in {0,...,3}
{
 \draw[very thick](-1,0.8*\x)--(6,0.8*\x);
\draw[very thick,dashed](6,0.8*\x)--(7,0.8*\x);
\draw[very thick,dashed](-2,0.8*\x)--(-1,0.8*\x);
 \foreach \y in {0,...,5}
  \filldraw(\y,0.8*\x)circle[radius=0.05];
}

\foreach \x in {0,...,5}
{
  \draw[opacity=0.3,fill=gray!80](\x-0.1,0-0.1)--(\x-0.1,2.4+0.1)--(\x+0.1,2.4+0.1)--(\x+0.1,0-0.1)--cycle;
}
\end{tikzpicture}

\end{center}
\caption{\footnotesize The operator described in the remark  is visualized here for $N=3$. 
The operator $\tilde{\Delta}$ is the adjacency operator over the graph $\ZZ\times\{0,\cdots,3\}$ where the edges are denoted by the black lines. 
The shaded region denotes the support of the  projections.}\label{fig1}
\end{figure}

First observe that
$$\Hi_k=\{u\in \ell^2(\ZZ\times\{0,\cdots,N\}): u(x,y)=0 ~\forall x\in\ZZ,y\neq k\}$$
is $H^\omega$ invariant and $\{(H^\omega,\Hi_k)\}_{k=0}^N$ are all unitarily equivalent.
So any singular spectrum has multiplicity $N$.
When $\{\omega_{n}\}_{n}$ are i.i.d, there are results \cite{K,SK1,BS4} which shows that $(H^\omega,\Hi_0)$ has pure point spectrum (hence singular spectrum). 
It is easy to show that the matrix $G^\omega_{n,n}(z)$ is of the form $f(z)I$, where $f$ is a Herglotz function and $I$ is identity on $\CC^N$.
\item On the Hilbert space $\ell^2(\NN\times\NN)$ consider the operator
$$H^\omega=\tilde{\Delta}+\sum_{(n,m)\in\NN^2} \omega_{(n,m)}P_{(n,m)}$$
where 
$$(\tilde{\Delta}u)(x,y)=\left\{\begin{matrix} u(2,y) & x=1 \\ u(x+1,y)+u(x-1,y) & x\neq 1\end{matrix}\right.\qquad\forall (x,y)\in\NN\times\NN$$
and the projections $P_{(n,m)}$ are given by
$$P_{(n,m)}=\sum_{k=2^{n(m-1)}}^{2^{nm}-1} \proj{\delta_{(n,k)}}.$$

\begin{figure}[ht]
\begin{center}
\begin{tikzpicture}[scale=0.6]
\clip[use as bounding box](-0.2,-0.2)rectangle(5.2,7.4);

\foreach \x in {0,...,14}
{
 \draw[very thick](0,\x*0.5)--(4,\x*0.5);
\draw[very thick,dashed](4,\x*0.5)--(5,\x*0.5);
 \foreach \y in {0,...,5}
  \filldraw(\y,\x*0.5)circle[radius=0.05];
}

\pgfmathparse{int(0)}\let\m\pgfmathresult;
\foreach \n in {0,...,3}
{
 \pgfmathparse{int(2^((\m+1)*\n)-1)}\let\minval\pgfmathresult;
 \pgfmathparse{int(2^((\m+1)*(\n+1))-2)}\let\maxval\pgfmathresult;
 \draw[opacity=0.3,fill=gray!80](\m-0.1,\minval*0.5-0.1)--(\m+0.1,\minval*0.5-0.1)--(\m+0.1,\maxval*0.5+0.1)--(\m-0.1,\maxval*0.5+0.1)--cycle;
}

\pgfmathparse{int(1)}\let\m\pgfmathresult;
\foreach \n in {0,...,1}
{
 \pgfmathparse{int(2^((\m+1)*\n)-1)}\let\minval\pgfmathresult;
 \pgfmathparse{int(2^((\m+1)*(\n+1))-2)}\let\maxval\pgfmathresult;
 \draw[opacity=0.3,fill=gray!80](\m-0.1,\minval*0.5-0.1)--(\m+0.1,\minval*0.5-0.1)--(\m+0.1,\maxval*0.5+0.1)--(\m-0.1,\maxval*0.5+0.1)--cycle;
}

\pgfmathparse{int(2)}\let\m\pgfmathresult;
\foreach \n in {0,...,1}
{
 \pgfmathparse{int(2^((\m+1)*\n)-1)}\let\minval\pgfmathresult;
 \pgfmathparse{int(2^((\m+1)*(\n+1))-2)}\let\maxval\pgfmathresult;
 \draw[opacity=0.3,fill=gray!80](\m-0.1,\minval*0.5-0.1)--(\m+0.1,\minval*0.5-0.1)--(\m+0.1,\maxval*0.5+0.1)--(\m-0.1,\maxval*0.5+0.1)--cycle;
}

\pgfmathparse{int(3)}\let\m\pgfmathresult;
\foreach \n in {0}
{
 \pgfmathparse{int(2^((\m+1)*\n)-1)}\let\minval\pgfmathresult;
 \pgfmathparse{int(2^((\m+1)*(\n+1))-2)}\let\maxval\pgfmathresult;
 \draw[opacity=0.3,fill=gray!80](\m-0.1,\minval*0.5-0.1)--(\m+0.1,\minval*0.5-0.1)--(\m+0.1,\maxval*0.5+0.1)--(\m-0.1,\maxval*0.5+0.1)--cycle;
}

\pgfmathparse{int(4)}\let\m\pgfmathresult;
\foreach \n in {0}
{
 \pgfmathparse{int(2^((\m+1)*\n)-1)}\let\minval\pgfmathresult;
 \pgfmathparse{int(2^((\m+1)*(\n+1))-2)}\let\maxval\pgfmathresult;
 \draw[opacity=0.3,fill=gray!80](\m-0.1,\minval*0.5-0.1)--(\m+0.1,\minval*0.5-0.1)--(\m+0.1,\maxval*0.5+0.1)--(\m-0.1,\maxval*0.5+0.1)--cycle;
}

\end{tikzpicture}

\end{center}
\caption{\footnotesize The operator described in the remark  is visualized here. 
The operator $\tilde{\Delta}$ is the adjacency operator over the graph $\NN^2$ where the edges are denoted by the black lines. The shaded region denotes the support of the  projections.}\label{fig2}
\end{figure}

In this example $P_{(n,m)}(H^\omega-z)^{-1}P_{(n,m)}$ is diagonal (w.r.t. the Dirac basis $\{\delta_{(n,m)}:n,m\in\NN\}$), and it is easy to see that
$$\sup_{(n,m)\in\NN} \mathcal{M}_{(n,m)}=\infty.$$
Similar to previous example the subspace
$$\Hi_k=\{u\in\ell^2(\NN\times\NN): u(x,y)=0~\forall x\in\NN,y\neq k\}\qquad\forall k\in\NN,$$
are invariant under the action of $H^\omega$. 
Notice that $\{(H^\omega,\Hi_k)\}_{k=2^m}^{2^{m+1}-1}$ are unitarily equivalent with each other for any $m\in\NN$. 
So the singular spectrum of $H^\omega$ has infinite multiplicity.
\end{enumerate}
So the conclusion of the theorem is optimal in the sense that there are random operators $A^\omega$ such that the multiplicity of singular spectrum is  $\sup_{n\in\Nc} \mathcal{M}_{n}$.
\end{remark}

The main technique in the proof involves studying the behavior of singular spectrum because of perturbation by single non-negative operator.
This is done through resolvent identity and so properties of matrix valued Herglotz functions plays an essential role. 
The steps involved in the proof will be further explained in section \ref{subsec1}.
In general these kind of results fails to hold without perturbation and spectral averaging \cite[Corollary 4.2]{CH1} plays an important role.
Since matrix valued Herglotz functions are the primary tool,  Poltoratskii's theorem \cite{POL1} is used to obtain and characterize the singular measure.

It should be noted that our result (Theorem \ref{mainThm}) extends the work of Jak{\v{s}}i{\'c}-Last \cite{JL1,JL2}, 
Naboko-Nichols-Stolz \cite{NNS} and Mallick \cite{AM3} in the following way:
in case of Jak{\v{s}}i{\'c}-Last\cite{JL1,JL2}, since the rank of each $P_n$ are one, above theorem gives the simplicity of singular spectrum. 
Naboko-Nichols-Stolz \cite{NNS} showed simplicity of the point spectrum for certain classes of Anderson type operator on $\ZZ^d$ 
and Mallick \cite{AM3} provided a bound on the multiplicity of the singular spectrum for a similar class of Anderson type operator on $\ZZ^d$.
In general it is not possible to compute $G^\omega_{n,n}(z)$, and so other methods has to be devised to get $\mathcal{M}_n$. 
The following corollary is a possible way to bound $\mathcal{M}_n$ for certain classes of random operators.

\begin{corollary}\label{corMultCouEx1}
On a separable Hilbert space $\Hi$, let $A^\omega$ defined by \eqref{eqn::mainOpEq1} satisfy the hypothesis of theorem \ref{mainThm}. 
Let $range(C_n)\subset \mathcal{D}(A)$ for all $n\in\Nc$, and let $M\in\RR$ be such that $\sigma(A)$ and $\sigma(A^\omega)$ are subset of $(M,\infty)$ for almost all $\omega$. 
Then 
\begin{enumerate}
\item If $C_n$ is a finite rank projection for all $n$, then the multiplicity of singular spectrum for $A^\omega$ is bounded above by 
$$\max_{n\in\Nc} \max_{x\in\sigma(C_nAC_n)} dim(ker(C_nAC_n-xI)),$$
where $C_nAC_n$ is viewed as a linear operator on $P_n\Hi$.
\item If $C_n$ is a non-negative finite rank operator for all $n$, then multiplicity of singular spectrum for $A^\omega$ is bounded above by 
$$\max_{n\in\Nc} \max_{x\in\sigma(C_n)} dim(ker(C_n-xI)),$$
where $C_n$ is viewed as a linear operator on $P_n\Hi$.
\end{enumerate}
\end{corollary}
\begin{remark}
It should be noted that the above bound is not optimal, but in many cases can be computed easily.
As an example, for the case of remark \ref{rem1} (1), 
all we have to do is count the eigenvalue multiplicity of $\chi_{S_r}\Delta \chi_{S_r}$, where $S_r=\{x\in\ZZ^d:\norm{x}_\infty=r\}$.
For $d=2$, this operator is same as the Laplacian on a set of $8n$ points arranged on a circle. 
So the multiplicity of the operator can be at most two. 
Another simple example is for the case when $C_n$ has simple spectrum, then the singular spectrum of $A^\omega$ is almost surely simple.
\end{remark}

The corollary should be considered as a generalization of the technique developed in Naboko-Nichols-Stolz \cite{NNS}. 
There the authors used the simplicity of $P_n\Delta P_n$ to conclude the simplicity of pure point spectrum for certain type of Anderson operators on $\ell^2(\ZZ^d)$.
Another similar work is \cite{AM3}, where the author bounded $\mathcal{M}_n$ by considering first few terms of Neumann series while keeping track of the perturbation.

Using an approach similar to \cite{AM3}, we can show that the singular spectrum for Anderson type operator on Bethe lattice is simple. 
Let $\mathcal{B}=(V,E)$ denote the infinite tree with root $e$ where each vertex has $K$ neighbors. Set $K> 2$ so that the tree is not isomorphic to $\ZZ$. 
Define the class of random operators
\begin{equation}\label{eqn::CanOpEq1}
H^\omega=\Delta_{\mathcal{B}}+\sum_{x\in J} \omega_x \chi_{\tilde{\Lambda}(x)}
\end{equation}
where $\Delta_{\mathcal{B}}$ is the adjacency operator on $\mathcal{B}$, and
$$\tilde{\Lambda}(x)=\{y\in V: d(e,x)\leq d(e,y)~\&~d(x,y)< l_x\},$$
for some $l_\cdot:V\rightarrow\NN$. Finally the indexing set $J\subset V$ be such that $\cup_{x\in J}\tilde{\Lambda}(x)=V$ and 
$$\tilde{\Lambda}(x)\cap \tilde{\Lambda}(y)=\phi\qquad\forall x\neq y \in J.$$
The random variables $\{\omega_x\}_{x\in J}$ are independent real valued with absolutely continuous distribution.
With these notation we have:
\begin{theorem}\label{thmSimBetheOp}
On a Bethe lattice $\mathcal{B}$ with $K>2$, consider a family of random operators $H^\omega$ given by \eqref{eqn::CanOpEq1}, 
where $\{\omega_x\}_{x\in J}$ are i.i.d random variables following absolutely continuous distribution with bounded support. 
Then singular spectrum of $H^\omega$ is almost surely simple.
\end{theorem}

It can be seen that the spectrum of $\chi_{\tilde{\Lambda}(x)}\Delta_{\mathcal{B}}\chi_{\tilde{\Lambda}(x)}$ has non-trivial multiplicity 
(is exponential in terms of the diameter of $\tilde{\Lambda}(x)$). 
So, the above result is not a consequence of previous corollary. 

\subsection{Structure of the Proof}\label{subsec1}
Rest of the article is divided into four parts. In section \ref{sec2}, we setup the notations and collect the results that will be used throughout. 
Section \ref{sec3} deals with single perturbation results. 
Section \ref{sec4} contains the proof of Theorem \ref{mainThm}, which is divided into Lemma \ref{lem::MultBou} and Lemma \ref{lem::MultEquiBd}. 
Finally in section \ref{sec5}, we prove the Corollary \ref{corMultCouEx1} and Theorem \ref{thmSimBetheOp}.

The proof of Theorem \ref{mainThm} is divided into three parts. 
First we concentrate on the operator $H_\lambda:=H+\lambda C$, where $H$ is a densely defined essentially self adjoint operator and $C$ is a finite rank non-negative operator. 
Since all the results are obtained through properties of Borel-Stieltjes transform, 
there is a set $S\subset\RR$, independent of $\lambda$, of full Lebesgue measure where all the analysis will be done.
As a consequence of spectral averaging (see Lemma \ref{lem::SpecAva}), it is enough to concentrate on $S$ as long as we are working on the subspace
$$\Hi^{\lambda}_{C}=\clsp{f(H_\lambda)\phi:f\in C_c(\RR)~\&~\phi\in C\Hi}.$$
By spectral averaging, the spectrum of $H_\lambda$ restricted to $\Hi^{\lambda}_{C}$ is contained in $S$ for almost all $\lambda$.
In section \ref{sec3}, we establish a certain inclusion relation between singular subspaces. 
We show that for any finite rank projection $Q$, the closed $H_\lambda$-invariant Hilbert subspace $\tilde{\Hi}^\lambda_Q\subseteq \Hi^\lambda_Q$, 
such that the spectrum of $H_\lambda$ restricted to $\tilde{\Hi}^\lambda_Q$ is singular and is contained in $S$, is a subset of singular subspace of $\Hi^\lambda_C$.
This inclusion is shown in Lemma \ref{lem::SingInc1}. This is the reason that the multiplicity of the singular subspace for $ \Hi^\omega_{\sum_{i=1}^N P_{n_i}}$ does not depend on $N$. 
Lemma \ref{lem::MultEquiBd} uses this fact to get a bound on the multiplicity of singular spectrum for $ \Hi^\omega_{\sum_{i=1}^N P_{n_i}}$ for any finite collection of $\{n_i\}_i$.
Finally global bound on the multiplicity of singular spectrum is obtained by observing the fact that $\cup_{N\in\NN}\Hi^\omega_{\sum_{i=1}^N P_{n_i}}$ is dense for any enumeration of $\Nc$.

Lemma \ref{lem::MultBou} provides the first conclusion of the theorem and also provides the relationship between $\mathcal{M}_n$ and multiplicity of singular spectrum for $\Hi^\omega_{P_n}$. 
The proof is mostly a consequence of properties of polynomial algebra where the coefficients of the polynomial under consideration are holomorphic function on $\CC\setminus\RR$.
Part of the work is to establish a relation between multiplicity of singular spectrum and multiplicity of $\sqrt{C_n}G^\omega_{n,n}(z)\sqrt{C_n}$, which is achieved through resolvent equation. 
After choosing a basis, we end up with matrix equations over function which are holomorphic on $\CC\setminus\RR$.
Since we are only dealing with matrices, multiplicity of $\sqrt{C_n}G^\omega_{n,n}(z)\sqrt{C_n}$ can be computed through its characteristic equation and so we have polynomial equations where the coefficients are polynomials of the matrix elements.
Most of the work is to show that it is independent of a single perturbation. 
Above argument also proves the independence from $z$, this is because the matrix elements are holomorphic functions on $\CC\setminus\RR$,
and so any non-zero polynomial can be zero only on a Lebesgue measure zero set. 
Then by induction we show that $Mult^\omega_n(z)$ is independent of any finite collection of random variables $\{\omega_{p_i}\}_i$.
Then Kolmogorov 0-1 law provides the stated result. 

Finally in section \ref{sec5}, we prove Corollary \ref{corMultCouEx1} and Theorem \ref{thmSimBetheOp}. This is mostly done by writing the matrix $G^\omega_{n,n}(z)$ into a particular form.
For the corollary, using the fact that $range(C_n)\subset \mathcal{D}(A)$, the matrix $C_n^{-\frac{1}{2}}AC_n^{-\frac{1}{2}}$ is well defined over $P_n\Hi$, and we have to estimate the number of eigenvalues of 
$$C_n^{-\frac{1}{2}}AC_n^{-\frac{1}{2}}+\mu C^{-1}$$ 
which are at most $O(1/\mu)$ distance away from each other, for $\mu\gg 1$. The corollary just deals with two extreme cases.
For Theorem \ref{thmSimBetheOp}, most of the work is to show that for a tree (of finite depth), the adjacency operator perturbed at all the leaf nodes has simple spectrum. 
Then the particular structure of $G^\omega_{n,n}(z)$ provides the conclusion.

Even though $G^\omega_{n,m}(z)$ are defined over $\CC\setminus\RR$, part of the proof of Lemma \ref{lem::SingInc1} is done on $\CC^{+}$ itself. 
The main problem that can arise on restricting to $\CC^{+}$ is because of F. and R. Riesz theorem \cite{RR}. 
It states that \emph{ if the Borel-Stieltjes transform of a measure is zero on $\CC^{+}$ then the measure is equivalent to Lebesgue measure} (see \cite[Theorem 2.2]{JL2} for a proof). 
This problem is avoided by using the fact that in case $G^\omega_{n,m}(z)$ is zero for $z\in\CC^{+}$, 
we can repeat the proof by switching to $z\in\CC^{-}$ and can replace $E+\iota \epsilon$ by $E-\iota\epsilon$ whenever necessary. 

\section{Preliminaries}\label{sec2}
In this section we setup the notations and results used in the rest of the work. 
Mostly we will deal with the linear operator 
$$G^\omega_{n,m}(z):= P_n(A^\omega-z)^{-1}P_m: P_m\Hi\rightarrow P_n\Hi\qquad \forall n,m\in \Nc,$$
which is well defined because of the assumption that $A^\omega$ is essentially self adjoint. Here $P_n$ denotes the orthogonal projection onto the range of $C_n$. 
We will denote 
$$\Hi^\omega_{P_n}:=\clsp{f(A^\omega)\phi: f\in C_c(\RR)~\&~\phi\in P_n\Hi }$$
to be the minimal closed $A^\omega$-invariant subspace containing $P_n\Hi$.
All the results are stated in a basis independent form, but sometimes explicit basis is fixed so that $G^\omega_{n,m}(z)$ can be viewed as a matrix valued functions.

We mostly focus on a single perturbation, which will be done as follows. For $p\in\Nc$ set $A^{\omega,\lambda}_p=A^\omega+\lambda C_p$ and define
$$G^{\omega,\lambda}_{p,n,m}(z)=P_n(A^{\omega,\lambda}_p-z)^{-1}P_m$$
as before. Using resolvent equation we have
\begin{align}
G^{\omega,\lambda}_{p,p,p}(z)&=G^\omega_{p,p}(z)(I+\lambda C_pG^\omega_{p,p}(z))^{-1},\label{eqn:resEq1}\\
G^{\omega,\lambda}_{p,n,m}(z)&=G^\omega_{n,m}(z)-\lambda G^\omega_{n,p}(z)(I+\lambda C_p G^\omega_{p,p}(z))^{-1}C_pG^\omega_{p,m}(z).\label{eqn:resEq2}
\end{align}
Another way to write above equations is
\begin{align}
&(I-\lambda C_pG^{\omega,\lambda}_{p,p,p}(z))(I+\lambda C_p G^\omega_{p,p}(z))=I,\label{eqn:resEq3}\\
&G^{\omega,\lambda}_{p,n,m}(z)=G^\omega_{n,m}(z)-\lambda G^\omega_{n,p}(z)C_pG^\omega_{p,m}(z)\nonumber\\
&\qquad\qquad\qquad\qquad+\lambda^2G^\omega_{n,p}(z)C_pG^{\omega,\lambda}_{p,p,p}(z)C_pG^\omega_{p,m}(z). \label{eqn:resEq4} 
\end{align}
Either of them will be used depending on the situation. 
It should be noted that the identity operator in equations \eqref{eqn:resEq1}, \eqref{eqn:resEq2} and \eqref{eqn:resEq3} is the identity map on $P_p\Hi$.
For a fixed basis of each of $P_n\Hi$, using \cite[Proposition 2.1]{JL1} (which follows from the property of Borel-Stieltjes transform) for each matrix elements of $G^\omega_{n,m}(z)$, we have that
$$G^\omega_{n,m}(E\pm\iota 0):=\lim_{\epsilon\downarrow 0}G^\omega_{n,m}(E\pm\iota \epsilon)$$
exists for almost all $E$ with respect to Lebesgue measure and for any $n,m\in\Nc$. So the linear operator $G^\omega_{n,m}(E\pm\iota 0)$ is well defined for almost all $E$ and any $n,m\in\Nc$. 

Using \eqref{eqn:resEq3} we observe that for any $E\in\RR$ such that $G^\omega_{p,p}(E\pm\iota0)$ exists and for $f:(0,\infty)\rightarrow\CC$ satisfying $\lim_{\epsilon\downarrow0}f(\epsilon)=0$, we have
\begin{align*}
\lim_{\epsilon\downarrow0}f(\epsilon)(I-\lambda C_p G^{\omega,\lambda}_{p,p,p}(E\pm\iota\epsilon))(I+\lambda C_p G^\omega_{p,p}(E\pm\iota\epsilon))=0\\
\Rightarrow\qquad   \left(\lim_{\epsilon\downarrow0}f(\epsilon) C_p G^{\omega,\lambda}_{p,p,p}(E\pm\iota\epsilon)C_p\right)(C_p^{-1}+\lambda  G^\omega_{p,p}(E\pm\iota0))=0,
\end{align*}
and similarly 
$$(C_p^{-1}+\lambda  G^\omega_{p,p}(E\pm\iota0))  \left(\lim_{\epsilon\downarrow0} f(\epsilon) C_p G^{\omega,\lambda}_{p,p,p}(E\pm\iota\epsilon)C_p \right)=0.$$
The above equation implies
\begin{align}
range\left( \left(\lim_{\epsilon\downarrow0}f(\epsilon) C_p G^{\omega,\lambda}_{p,p,p}(E\pm\iota\epsilon)C_p\right)\right)&\subseteq ker(C_p^{-1}+\lambda  G^\omega_{p,p}(E\pm\iota0))\nonumber\\
&\subseteq ker(\Im G^\omega_{p,p}(E\pm\iota0) ),\label{eqn::kerEq1}
\end{align}
which is used to determine the singular spectrum. 
One of the consequences of $\pm\Im G_{p,p}^\omega(E\pm\iota 0)\geq 0$ is
\begin{equation}\label{eqn::resEq5}
G^\omega_{k,p}(E\pm\iota 0)\phi=G^\omega_{p,k}(E\pm\iota 0)^\ast\phi\qquad \forall \phi\in ker(\pm\Im G_{p,p}^\omega(E\pm\iota 0)),
\end{equation}
which plays an important role in the proof of Lemma \ref{lem::SingInc1}.

Since most of the analysis is done using a single perturbation, one of the important results needed is the spectral averaging; 
we refer to \cite[Corollary 4.2]{CH1} for its proof. Here we will use the following version:
\begin{lemma}\label{lem::SpecAva}
Let $E_\lambda(\cdot)$ be the spectral family for the operator $A_\lambda=A+\lambda C$, where $A$ is a self adjoint operator and $C$ is a non-negative compact operator. 
For any $M\subset\RR$ with zero Lebesgue measure, we have  $\sqrt{C}E_\lambda(M)\sqrt{C}=0$ for almost all $\lambda$, with respect to Lebesgue measure.
\end{lemma}
Since the set of $E$ where $\lim_{\epsilon\downarrow0}G^\omega_{n,m}(E\pm\iota\epsilon)$ does not exists for some $n,m\in\Nc$, is a Lebesgue measure zero set, 
above lemma guarantees that we can leave this set from our analysis as long as we are only focusing on $A_p^{\omega,\lambda}$-invariant subspace containing $P_p\Hi$. 
Another important result is 
\begin{lemma}\label{lem::zeroMeaSet}
 For a $\sigma$-finite positive measure space $(X,\mathscr{B},m)$ and a collection of $\mathscr{B}$-measurable functions $a_i:X\rightarrow\CC$ and $b_i:X\rightarrow\CC$, define $$f(\lambda)=\frac{1+\sum_{n=1}^N a_n(x)\lambda^n}{1+\sum_{n=1}^N b_n(x)\lambda^n},$$
 Then the set
 $$\Lambda_f=\{\lambda\in\CC: m(x\in X: f(\lambda,x)=0)>0\}$$
 is countable.
\end{lemma}
Its proof can be found in \cite[Lemma 2.1]{AM1}. 
This lemma ensures that the linear operator $G^{\omega,\lambda}_{p,p,p}(z)$ is well defined for almost all $\lambda$. 
This is the case because $G^{\omega,\lambda}_{p,p,p}(z)$ and $G^{\omega}_{p,p}(z)$ are related through the equation \eqref{eqn:resEq1}, 
and so the set $\{E:\det(I+\lambda C_p G^\omega_{p,p}(E\pm\iota 0))=0\}$ should have zero Lebesgue measure, otherwise the analysis will fail. 
This is also the set in which the singular spectrum of $A^{\omega,\lambda}_p$ restricted to $\Hi^{\omega}_{P_p}$ (it is easy to see that the space$\Hi^\omega_{P_p}$ is invariant under action of $A^{\omega,\lambda}_p$) belongs. 

Next result is Poltoratskii's theorem and is the main tool through which singular part of the spectrum is handled.
Since we only deal with finite measures, we will denote the Borel-Stieltjes transform $F_\mu:\CC^{+}\rightarrow\CC^{+}$ for the Borel measure $\mu$ by
$$F_\mu(z)=\int \frac{d\mu(x)}{x-z}.$$
For $f\in L^1(\RR,d\mu)$, let $f\mu$ be the unique measure associated with the linear functional $C_c(\RR)\ni g\mapsto \int g(x)f(x)d\mu(x)$.
The version of the Poltoratskii's theorem we will use is:
\begin{lemma}\label{lem::polThm}
For any complex valued Borel measure  $\mu$ on $\RR$, let $f\in L^1(\RR,d\mu)$, then
$$\lim_{\epsilon\downarrow0}\frac{F_{f\mu}(E+\iota\epsilon)}{F_\mu(E+\iota\epsilon)}=f(E)$$
for a.e $E$ with respect to $\mu$-singular. 
\end{lemma}
The proof of this can be found in \cite{JL3}. With these results in hand, we can now prove our results.
\section{Single Perturbation Results}\label{sec3}
This section will concentrate on a single perturbation. Lemma \ref{lem::SingInc1} will play an important role for proving the main result. 
For this section a different notation will be followed, because it is not necessary to keep track of all the random variables $\{\omega_n\}_n$.

 Let $H$ be a densely defined self adjoint operator on a separable Hilbert space $\Hi$ and $C_1$ be a finite rank non-negative operator. 
Set $H_\lambda=H+\lambda C_1$ and let $P_1$ be the orthogonal projection onto the range of $C_1$. 
For any projection $Q$ define
$$\Hi^\lambda_{Q}:=\clsp{f(H_\lambda)\psi: \psi\in Q\Hi~\&~f\in C_c(\RR)},$$
to be the minimal closed $H_\lambda$-invariant subspace containing the range of $Q$. 
Let $\sigma^\lambda_1$ denote the trace measure $tr(P_1 E^{H_\lambda}(\cdot))$, where $E^{H_\lambda}(\cdot)$ is the spectral projection for the operator $H_\lambda$.
The subscript $sing$ will be used to denote the singular part of the measure whenever necessary.
The main result of this section is the following:
\begin{lemma}\label{lem::SingInc1}
Let $Q$ be a finite rank projection and set $\{e_i\}_{i}$ to be an orthonormal basis of $Q\Hi+P_1\Hi$. Define the set
$$S=\{E\in\RR: \dprod{e_i}{(H-E\mp\iota 0)^{-1}e_j}\text{ exists and finite}\},$$
and denote $E^\lambda_{sing }$ to be the spectral measure onto the singular part of spectrum of $H_\lambda$, then
$$E^\mu_{sing }(S)\Hi^\lambda_{Q}\subseteq E^\mu_{sing }(S)\Hi^\lambda_{P_1}$$
for almost all $\lambda$ with respect to the Lebesgue measure.
\end{lemma}
\begin{remark}\label{rem4}
Spectral averaging result (Lemma \ref{lem::SpecAva}) gives $\sigma^\lambda_1(\RR\setminus S)=0$ for a.a. $\lambda$ w.r.t. Lebesgue measure, 
so it is actually not necessary to write $S$ on RHS of above equation. But $E^\lambda_{sing}(\RR\setminus S)\Hi^\lambda_{Q}$ can be non-trivial.
\end{remark}
\begin{proof}
In view of Lemma \ref{lem::SumHilSub}, it is enough to show 
$$E^\lambda_{sing }(S)\Hi^\lambda_{e_i}\subseteq E^\lambda_{sing }(S)\Hi^\lambda_{P_1},$$
where $\Hi^\lambda_{e_i}$ is the minimal closed $H_\lambda$-invariant subspaces containing $e_i$. 
This is because applying Lemma \ref{lem::SumHilSub} for the operator $E^\lambda_{sing }(S) H_\lambda$ will give the singular subspaces in the conclusion of the lemma.

Using the resolvent equation 
$$(H_\lambda-z)^{-1}-(H-z)^{-1}=-\lambda (H_\lambda-z)^{-1}C_1(H-z)^{-1}$$
and similarly
\begin{align*}
(H_\lambda-z)^{-1}&=(H-z)^{-1}-\lambda (H-z)^{-1}C_1(H_\lambda-z)^{-1}\\
&=(H-z)^{-1}-\lambda (H-z)^{-1}C_1(H-z)^{-1}\\
&\qquad+\lambda^2 (H-z)^{-1}C_1(H_\lambda-z)^{-1}C_1(H_\lambda-z)^{-1},
\end{align*}
we have
\begin{align}\label{eqn:SingHilInc1Eq1}
\dprod{e_i}{(H_\lambda-z)^{-1}e_i}&=\dprod{e_i}{(H-z)^{-1}e_i}-\lambda \dprod{e_i}{(H-z)^{-1}C_1(H-z)^{-1}e_i}\nonumber\\
&\qquad +\lambda^2 \dprod{e_i}{(H-z)^{-1}C_1(H_\lambda-z)^{-1}C_1(H-z)^{-1}e_i}.
\end{align}
Let $\{e_{1i}\}_{i=1}^{r_1}$, where $r_1=dim(P_1\Hi)$, be an orthonormal basis of $P_1\Hi$ (so that they are linear combinations of $\{e_i\}_i$); 
hence $G^\lambda_{1,1}(z)=P_1(H_\lambda-z)^{-1}P_1$ is a matrix for this basis and also set
$$G_{i,1}(z)=\matx{\dprod{e_i}{(H-z)^{-1}e_{11}}\\ \dprod{e_i}{(H-z)^{-1}e_{12}}\\ \vdots \\ \dprod{e_i}{(H-z)^{-1}e_{1r_1}} }^t~~\&~~  G_{1,i}(z)=\matx{ \dprod{e_{11}}{(H-z)^{-1}e_i}\\ 
\dprod{e_{12}}{(H-z)^{-1}e_i}\\ \vdots \\ \dprod{e_{1r_1}}{(H-z)^{-1}e_i} } .$$
Then the equation \eqref{eqn:SingHilInc1Eq1} can be written as
\begin{align*}
\dprod{e_i}{(H_\lambda-z)^{-1}e_i}&=\dprod{e_i}{(H-z)^{-1}e_i}-\lambda G_{i,1}(z)C_1 G_{1,i}(z)\\
&\qquad+\lambda^2 G_{i,1}(z)C_1 G_{1,1}^\lambda(z)C_1G_{1,i}(z).
\end{align*}
Using the fact that LHS is the Borel-Stieltjes transform of the measure \\
$\langle e_i E^{H_\lambda}( \cdot )e_i\rangle$, the support of singular part lies in the set of $E\in\RR$ where 
$$\lim_{\epsilon\downarrow 0} \left(\dprod{e_i}{(H_\lambda-E-\iota \epsilon)^{-1}e_i}\right)^{-1}=0.$$
We don't need to consider the case $\dprod{e_i}{(H_\lambda-z)^{-1}e_i}= 0$ for all $z\in\CC^{+}$ because by F. and R. Riesz theorem \cite{RR}, 
the measure $\dprod{e_i}{E^{H_\lambda}(\cdot)e_i}$ is absolutely continuous.
But by definition of the set $S$, we have $G_{i,1}(E\pm\iota 0),G_{1,i}(E\pm\iota 0)$ and $\dprod{e_i}{(H-E\mp\iota 0)^{-1}e_i}$ exist for each $E\in S$. 
So singular part of $\dprod{e_i}{E^{H_\lambda}(\cdot)e_i}$ can lie on $\RR\setminus S$ or on the set of $E\in S$ where $\lim_{\epsilon\downarrow 0}(tr(G_{1,1}^\lambda(E+\iota \epsilon)))^{-1}=0$. 
For $E\in S$ where $\lim_{\epsilon\downarrow 0}(tr(G_{1,1}^\lambda(E+\iota \epsilon)))^{-1}=0$, note that
\begin{align*}
&\lim_{\epsilon\downarrow 0}\frac{\dprod{e_i}{(H_\lambda-E-\iota \epsilon)^{-1}e_i}}{tr(G_{1,1}^\lambda(E+\iota \epsilon))}\\
&\qquad=\lambda^2 G_{i,1}(E+\iota 0)C_1\left(\lim_{\epsilon\downarrow 0} \frac{ G_{1,1}^\lambda(E+\iota\epsilon)}{tr(G_{1,1}^\lambda(E+\iota \epsilon))}\right)C_1G_{1,i}(E+\iota 0).
\end{align*}
Using \eqref{eqn::resEq5}, we have
\begin{align}\label{eqn::SingHilInc1Eq2}
&\lim_{\epsilon\downarrow 0}\frac{\dprod{e_i}{(H_\lambda-E-\iota \epsilon)^{-1}e_i}}{tr(G_{1,1}^\lambda(E+\iota \epsilon))}\nonumber\\
&\qquad=\lambda^2 [C_1 G_{1,i}(E+\iota 0)]^\ast \left(\lim_{\epsilon\downarrow 0} \frac{ G_{1,1}^\lambda(E+\iota\epsilon)}{tr(G_{1,1}^\lambda(E+\iota \epsilon))}\right)[C_1G_{1,i}(E+\iota 0)].
\end{align}

Since $G_{1,1}^\lambda(\cdot)$ is a matrix valued Herglotz function for a positive operator valued measure (it is the Borel transform of $P_1 E^{H_\lambda}(\cdot)P_1$), 
there exists a matrix valued function $M_1^\lambda\in L^1(\RR,\sigma_1^\lambda,M_{rank(P_1)}(\CC))$, (using the Herglotz representation theorem for matrix valued measures, see \cite[Theorem 5.4]{GT1}) 
such that we have
$$G_{1,1}^\lambda(z)=\int \frac{1}{x-z}M_1^\lambda(x)d\sigma^\lambda_1(x),$$
for $z\in\CC\setminus\RR$. Using Poltoratskii's theorem (lemma \ref{lem::polThm}) we have
$$\lim_{\epsilon\downarrow 0}\frac{1}{tr(G_{1,1}^\lambda(E+\iota\epsilon))}G_{1,1}^\lambda(E+\iota\epsilon)=M_1^\lambda(E)$$
for almost all $E$ with respect to $\sigma^\lambda_{1,sing}$. 
Since the measure $P_1E^{H_\lambda}(\cdot)P_1$ is non-negative, the matrix valued function $M_1^\lambda(E)\geq 0$ for almost all $E$ with respect to $\sigma^\lambda_1$.

Let $U_1^\lambda(E)$ be the unitary matrix which diagonalizes $M_1^\lambda(E)$, i.e
$$U_1^\lambda(E)M_1^\lambda(E)U_1^\lambda(E)^\ast=diag(f^\lambda_j;1\leq j\leq r_1),$$
where some of the $f^\lambda_j$ can be zero. 
Using Hahn-Hellinger theorem (see \cite[Theorem 1.34]{N1}), the function $U_i^\lambda$ can be chosen to be a Borel measurable Unitary matrix valued function. 
Since we only focus on singular part, set $U_1^\lambda(E)=0$ for $E$ not in support of $\sigma^\lambda_{1,sing}$ and define $\psi^\lambda_j=U_1^\lambda(H_\lambda)^\ast e_{1j}$. Now observe that
\begin{align*}
&\dprod{\psi^\lambda_k}{(H_\lambda-z)^{-1}\psi^\lambda_l}=\int \frac{1}{x-z} \dprod{\psi^\lambda_k}{E^{H_\lambda}(dx)\psi^\lambda_l}\\
&\qquad=\int \frac{1}{x-z} \dprod{U_1^\lambda(x)^\ast e_{1k}}{E^{H_\lambda}(dx)U_1^\lambda(x)^\ast e_{1l}}\\
&\qquad=\int \frac{1}{x-z} \sum_{p,q}\dprod{U_1^\lambda(x)^\ast e_{1k}}{e_{1p}}\dprod{e_{1q}}{U_1^\lambda(x)^\ast e_{1l}}\dprod{e_{1p}}{E^{H_\lambda}(dx) e_{1q}}\\
&\qquad=\sum_{p,q}\int \frac{1}{x-z}\dprod{U_1^\lambda(x)^\ast e_{1k}}{e_{1p}}\dprod{e_{1q}}{U_1^\lambda(x)^\ast e_{1l}}\dprod{e_{1p}}{E^{H_\lambda}(dx) e_{1q}},
\end{align*}
and so using Poltoratskii's theorem (lemma \ref{lem::polThm}) we get
\begin{align*}
&\lim_{\epsilon\downarrow 0}\frac{\dprod{\psi^\lambda_k}{(H_\lambda-E-\iota\epsilon)^{-1}\psi^\lambda_l}}{tr(G_{1,1}^\lambda(E+\iota\epsilon))}\\
&\qquad=\sum_{p,q}\dprod{U_1^\lambda(E)^\ast e_{1k}}{e_{1p}}\dprod{e_{1q}}{U_1^\lambda(E)^\ast e_{1l}}\left(\lim_{\epsilon\downarrow 0}\frac{\dprod{e_{1p}}{(H_\lambda-E-\iota\epsilon)^{-1}e_{1q}}}{tr(G_{1,1}^\lambda(E+\iota\epsilon))}\right)\\
&\qquad=\dprod{e_{1k}}{U_1^\lambda(E)M_1^\lambda(E)U_1^\lambda(E)^\ast e_{1l}}=f^\lambda_k(E)\delta_{k,l}
\end{align*}
for almost all $E$ with respect to $\sigma^\lambda_{1,sing}$. 
By construction of $\psi^\lambda_j$, the spectral measure $\dprod{\psi^\lambda_j}{E^{H_\lambda}(\cdot)\psi^\lambda_j}$ is purely singular with respect to the Lebesgue measure,
so above computation implies $\dprod{\psi^\lambda_k}{(H_\lambda-z)^{-1}\psi^\lambda_l}=0$ for all $z$ for $k\neq l$, which implies that the measure $\dprod{\psi^\lambda_k}{E^{H_\lambda}(\cdot)\psi^\lambda_l}$ is zero, 
and in particular we have $\Hi^\lambda_{\psi^\lambda_k}\perp \Hi^\lambda_{\psi^\lambda_l}$ for $k\neq l$. 

Next, using the resolvent equation, we obtain
\begin{align}\label{eqn::SingHilInc1Eq3}
&\lim_{\epsilon \downarrow 0}\frac{\dprod{\psi^\lambda_k}{(H_\lambda-E-\iota\epsilon)^{-1}e_i}}{tr(G_{1,1}^\lambda(E+\iota\epsilon))}\\
&\qquad=\lim_{\epsilon \downarrow 0}-\lambda \frac{\dprod{\psi^\lambda_k}{(H_\lambda-E-\iota\epsilon)^{-1}C_1(H-E-\iota\epsilon)^{-1} e_i}}{tr(G_{1,1}^\lambda(E+\iota\epsilon))}\nonumber\\
&\qquad= -\lambda f^\lambda_k(E)\dprod{e_{1k}}{U_1^\lambda(E)C_1 G_{1,i}(E+\iota 0)},\nonumber
\end{align}
for a.e. $E$ w.r.t. $\sigma^\lambda_{1,sing}$.  Using Lemma \ref{lem::ProjSub} and the above equation \eqref{eqn::SingHilInc1Eq3} on equation \eqref{eqn::SingHilInc1Eq2}, we conclude
$$\lim_{\epsilon\downarrow 0}\frac{\dprod{e_i}{(H_\lambda-E-\iota \epsilon)^{-1}e_i}}{tr(G_{1,1}^\lambda(E+\iota \epsilon))}=\sum_j \left|(Q_{\psi^\lambda_j}^\lambda e_i)(E)\right|^2 f_j^\lambda(E)$$
for a.e. $E$ w.r.t $\sigma^\lambda_{1,sing}$, where $Q_{\psi^\lambda_j}^\lambda e_i$ is the projection of $e_i$ on the Hilbert subspace $\Hi^\mu_{\psi^\lambda_j}$. 
So for $g\in C_c(\RR)$, we can write
\begin{align*}
\dprod{e_i}{E^\lambda_{sing}(S) g(H_\lambda)e_i}&=\sum_j \int g(E)\left|(Q_{\psi^\lambda_j}^\lambda e_i)(E)\right|^2 f_j^\lambda(E) d\sigma^\lambda_{1,sing}(E),
\end{align*}
which implies that the projection of $E^\lambda_{sing}(S)e_i$ onto $\Hi^\lambda_{P_1}$ is isometry, hence
$$E^\lambda_{sing}(S)\Hi_{e_i}^\lambda\subseteq E^\lambda_{sing}(S)\Hi^\lambda_{P_1}.$$
The lemma follows by an application of Lemma \ref{lem::SumHilSub}.

\end{proof}

\section{Proof of Theorem \ref{mainThm}}\label{sec4}
The proof of the main result is divided into Lemma \ref{lem::MultBou} and Lemma \ref{lem::MultEquiBd}.
It should be noted that the conclusion of Lemma \ref{lem::MultBou} is similar to the conclusion reached by combining \cite[lemma 2.2]{AM3} and \cite[lemma 2.1]{AM3}.
This section deals with $A^\omega$ itself and so the notations established in section \ref{sec2} are followed. 
Following the notations from previous section, set $\Hi^\omega_P$ to be the minimal closed $A^\omega$-invariant subspace containing the range of the projection $P$. 

\begin{lemma}\label{lem::MultBou}
For any $n\in\Nc$, 
$$\mathcal{M}^\omega_n:=\essSup{z\in\CC\setminus\RR} Mult^\omega_n(z)$$
is almost surely constant; denote it by $\mathcal{M}_n$. The multiplicity of singular spectrum for $\Hi^\omega_{P_n}$ is bounded above by $\mathcal{M}_n$.
\end{lemma}
\begin{proof}
First we prove that $\mathcal{M}_n^\omega$ is independent of $\omega$. This is done using Kolmogorov $0$-$1$ law. 
So first step is to show that $\mathcal{M}_n^\omega$ is independent of any finite collection of random variables $\{\omega_{p_i}\}_{i}$.

Following the notations from section \ref{sec2}, set $A^{\omega,\lambda}_p=A^\omega+\lambda C_p$ for $p\in\Nc\setminus\{n\}$, we have the equation \eqref{eqn:resEq2}
$$G^{\omega,\lambda}_{p,n,n}(z)=G^\omega_{n,n}(z)-\lambda G^\omega_{n,p}(z)(I+\lambda C_p G^\omega_{p,p}(z))^{-1}C_p G^\omega_{p,n}(z).$$
Looking at $G^\omega_{i,j}(z)$ as a matrix, observe that
\begin{align*}
\tilde{g}^\omega_{\lambda,z}(x)&=\det(C_n G^{\omega,\lambda}_{p,n,n}(z)-xI)\\
&=\det(C_n G^\omega_{n,n}(z)-\lambda C_n G^\omega_{n,p}(z)(I+\lambda C_p G^\omega_{p,p}(z))^{-1}C_p G^\omega_{p,n}(z)-xI)\\
&=\frac{p_l^\omega(z,\lambda)x^l+p_{l-1}^\omega(z,\lambda)x^{l-1}+\cdots+p_0^\omega(z,\lambda) }{\det(C_p^{-1}+\lambda G^\omega_{n,n}(z))},
\end{align*}
where $l=rank(P_n)$. Here $\{p_i^\omega(z,\lambda)\}_{i=0}^l$ are polynomials in the elements of the matrices $\{G^\omega_{i,j}(z)\}_{i,j\in\{n,p\}}$ and $\lambda$. 
We don't need to focus on the denominator, so set
$$g^\omega_{\lambda,z}(x)=p_l^\omega(z,\lambda)x^l+p_{l-1}^\omega(z,\lambda)x^{l-1}+\cdots+p_0^\omega(z,\lambda).$$
The maximum algebraic multiplicity of $G^{\omega,\lambda}_{p,n,n}(z)$ is $k$ if the function
$$\mathcal{F}^\omega_{\lambda,z}(x)=gcd\left(g^\omega_{\lambda,z}(x),\frac{dg^\omega_{\lambda,z}}{dx}(x),\cdots,\frac{d^kg^\omega_{\lambda,z}}{dx^k}(x)\right)$$
is constant with respect to $x$. Using  the fact that
$$gcd(f_1(x),\cdots,f_m(x))=gcd(f_1(x),\cdots,f_{m-2}(x),gcd(f_{m-1}(x),f_m(x)))$$
and Euclid's algorithm for polynomials, we get 
$$\mathcal{F}^\omega_{\lambda,z}(x)=q^\omega_0(\lambda,z)+q^\omega_1(\lambda,z)x+\cdots+q^\omega_s(\lambda,z)x^s$$
where $\{q^\omega_i(\lambda,z)\}_{i=0}^s$ are rational polynomials of $\{p_i^\omega(z,\lambda)\}_i$. 
We need to consider the numerators of $q_i^\omega$, which are denoted by $\tilde{q}_i^\omega$. 
Since $\{\tilde{q}_i^\omega\}$ are polynomials of the matrix elements $\{G^\omega_{i,j}(z)\}_{i,j\in\{n,p\}}$ and $\lambda$, write
$$\tilde{q}^\omega_i(\lambda,z)=\sum_j a^\omega_{ij}(z)\lambda^j$$
where $\{a^\omega_{ij}\}_{i,j}$ are holomorphic functions on $\CC\setminus\RR$. So $\{\tilde{q}^\omega_i\}$ are well defined over $(\lambda,z)\in \RR\times(\CC\setminus\RR)$ for each $i$.

Now suppose $\mathcal{M}^\omega_n=k$, then $q^\omega_0(0,\cdot)\neq 0$ and $q^\omega_i(0,\cdot)=0$ identically, which implies $a^\omega_{i0}(\cdot)=0$ for $i\neq 0$. 
This implies $G^{\omega,\lambda}_{p,n,n}(z)$ can have multiplicity greater than $k$.
Setting $\tilde{\omega}^p$ to be such that $\tilde{\omega}^p_k=\omega_k$ for $k\neq p$ and $\tilde{\omega}^p_p=\omega_p+\lambda$, gives $\mathcal{M}_n^\omega\leq \mathcal{M}_n^{\tilde{\omega}^p}$. 
Since $\mathcal{M}_n^{\tilde{\omega}^p}$ can be at most $rank(P_n)$, this implies $\mathcal{M}_n^{\tilde{\omega}^p}$ is independent of $\lambda$.

Now repeating the proof inductively for a collection of sites $\{p_i\}_{i=1}^N$ proves the independence of  $\mathcal{M}_n^{\omega}$ from the random variables $\{\omega_{p_i}\}_{i=1}^N$. 
Hence, using Kolmogorov 0-1 law, $\mathcal{M}_n^\omega$ is independent of $\omega$.

Assume that $\mathcal{M}_n=k$, which implies that the maximum multiplicity for the matrix $G^\omega_{n,n}(z)$ is $k$ for almost every $z$.  
Using above argument for the polynomial
$$g^\omega_z(x)=\det(C_n G^\omega_{n,n}(z)-xI)=(-x)^l+(-x)^{l-1}p_{l-1}^\omega(z)+\cdots+p_0^\omega(z),$$
we get that the function
\begin{align*}
gcd\left(g^\omega_{z}(x),\frac{dg^\omega_{z}}{dx}(x),\cdots,\frac{d^kg^\omega_{z}}{dx^k}(x)\right)
\end{align*}
is a rational polynomial of matrix elements of $G^\omega_{n,n}(z)$ and so the numerator is holomorphic on $\CC\setminus\RR$. 
Since it is non-zero for a positive Lebesgue measure set, it is non-zero for almost all $z\in\CC\setminus\RR$, which implies
\begin{align}\label{eqn::MultBouEq1}
 k=\essSup{E\in\RR} \{\text{Maximum multiplicity of roots of }\nonumber\\
 \det(C_n G_{n,n}^\omega(E\pm\iota 0)-xI)\}.
\end{align}

Now focus on the second conclusion of the Lemma, i.e. multiplicity of singular spectrum on $\Hi^\omega_{P_n}$ is bounded by $\mathcal{M}_n$.
Denote 
\begin{align}\label{eqn::MultBouEq2}
 S=\{E\in\RR: \text{Maximum multiplicity of roots of }\nonumber \\ 
 \det(C_n G_{n,n}^\omega(E\pm\iota 0)-xI) \text{ is }k\},
\end{align}
which by above has full Lebesgue measure.

Using Spectral theorem (see \cite[Theorem A.3]{AM1}) for the operator $A^{\omega,\lambda}_n=A^\omega+\lambda C_n$ gives
$$(\Hi^{\omega,\lambda,n}_{P_n},A^{\omega,\lambda}_n)\cong (L^2(\RR,P_n E^{A^{\omega,\lambda}_n}(\cdot)P_n,P_n\Hi),M_{Id}).$$
Here $E^{A^{\omega,\lambda}_n}$ is the spectral measure for $A^{\omega,\lambda}_n$ and $\Hi^{\omega,\lambda,n}_{Q}$ is the minimal closed $A^{\omega,\lambda}_n$-invariant space containing the subspace $Q\Hi$ for a projection $Q$.
Since the measure $P_n E^{A^{\omega,\lambda}_n}(\cdot)P_n$ is absolutely continuous with respect to the trace measure $\sigma_n^{\omega,\lambda}(\cdot)=tr(P_n E^{A^{\omega,\lambda}_n}(\cdot)P_n)$, 
after a choice of basis, there exists a non-negative $M_n^{\omega,\lambda}\in L^1(\RR,\sigma_n^{\omega,\lambda},M_{rank(P_n)}(\CC))$ such that
$$P_n E^{A^{\omega,\lambda}_n}(dx)P_n=M_n^{\omega,\lambda}(x)\sigma_n^{\omega,\lambda}(dx),$$
and Poltoratskii's theorem (lemma \ref{lem::polThm}) gives us that
$$\lim_{\epsilon\downarrow 0}\frac{1}{tr(G^{\omega,\lambda}_{n,n,n}(E+\iota\epsilon))}G^{\omega,\lambda}_{n,n,n}(E+\iota\epsilon)=M_n^{\omega,\lambda}(E)$$
for almost all $E$ with respect to $\sigma_n^{\omega,\lambda}$-singular. 
Here we are assuming that $\sigma_n^{\omega,\lambda}$ has a non-trivial singular component, so $G^{\omega,\lambda}_{n,n,n}(z)\neq 0$ for almost all $z\in\CC^{+}$. 
Just as in \eqref{eqn:resEq3} we  also have
$$(I+\lambda C_n G^\omega_{n,n}(z))(I-\lambda C_n G^{\omega,\lambda}_{n,n,n}(z))=I,$$
which gives (using steps involved for obtaining \eqref{eqn::kerEq1})
\begin{align*}
(I+\lambda C_n G^\omega_{n,n}(E+\iota 0)) \left[C_n\lim_{\epsilon\downarrow 0}\frac{1}{tr(G^{\omega,\lambda}_{n,n,n}(E+\iota\epsilon))} G^{\omega,\lambda}_{n,n,n}(E+\iota \epsilon)\right]=0,
\end{align*}
for $E$ whenever $\lim_{\epsilon\downarrow 0}\frac{1}{tr(G^{\omega,\lambda}_{n,n,n}(E+\iota\epsilon))}=0$. So 
$$(I+\lambda C_n G^\omega_{n,n}(E+\iota 0)) C_n M_n^{\omega,\lambda}(E)=0$$
for almost all $E$ with respect to $\sigma_n^{\omega,\lambda}$-singular. 
Using the fact that $\sigma_n^{\omega,\lambda}(\RR\setminus S)=0$ for almost all $\lambda$ and the above equation, 
which implies that the rank of $M_n^{\omega,\lambda}(E)$ is upper bounded by dimension of the kernel $(I+\lambda C_n G^\omega_{n,n}(E+\iota 0))$ 
which in turn is upper bounded by $k$ over the set $S$ (follows from \eqref{eqn::MultBouEq2}), 
we get that the multiplicity of the singular spectrum for $A^{\omega,\lambda}_n$ is bounded above by $k$ over $\Hi^{\omega,\lambda,n}_{P_n}$. 

This completes the proof as the above statement is true for almost all $(\omega,\lambda)$.

\end{proof}
Note that, in the above lemma  bound for the multiplicity of singular spectrum is given for the subspace $\Hi^\omega_{P_n}$ and not on the entire Hilbert space.
Lemma \ref{lem::SingInc1} is used to obtain the final result, which is as follows:
\begin{lemma}\label{lem::MultEquiBd}
Assuming the hypothesis of Theorem \ref{mainThm} and that $\mathcal{M}_n\leq K$ for all $n\in\Nc$. Then the multiplicity of singular spectrum for $A^\omega$ is bounded above by $K$ almost surely.
\end{lemma}
\begin{proof}
The proof is done in two steps. 
First we show that for any finite collections of $\{p_i\}_{i=1}^N\subset \Nc$, the multiplicity of singular spectrum restricted to $\Hi^\omega_{\sum_{i=1}^N P_{p_i}}$ is bounded by $K$.
Then the proof is completed using the density of $\cup_{N=1}^\infty \Hi^\omega_{\sum_{i=1}^N P_{p_i}}$.

First part is through induction, so let $\{p_i\}_{i\in\NN}$ be an enumeration of the set $\Nc$. 
The induction is done over the statement $\mathcal{S}_N$ which is: \emph{Multiplicity of Singular spectrum for $A^\omega$ restricted to the subspace $\Hi^\omega_{\sum_{i=1}^N P_{p_i}}$ is at most $K$}.

For the case $N=1$, the conclusion follows from the Lemma \ref{lem::MultBou}, i.e the multiplicity of the singular spectrum over $\Hi^\omega_{P_{p_1}}$ is at most $K$.

For the induction step assume $\mathcal{S}_N$ is true, i.e the multiplicity of the singular spectrum over $\Hi^\omega_{\sum_{i=1}^N P_{p_i}}$ is bounded by $K$. 
Before going in to prove $\mathcal{S}_{N+1}$, note that
$$\Hi^\omega_{\sum_{i=1}^{N+1} P_{p_i}}= \Hi^\omega_{\sum_{i=1}^N P_{p_i}}+\Hi^\omega_{P_{p_{N+1}}},$$
RHS is a subset of LHS is obvious, and for the other inclusion observe that RHS is dense and closed in LHS. 

Now consider the operator $A^{\omega,\lambda}_{p_{N+1}}=A^\omega+\lambda C_{p_{N+1}}$.
By Lemma \ref{lem::MultBou}, the multiplicity of singular spectrum for $A^{\omega,\lambda}_{p_{N+1}}$ over $\Hi^{\omega,\lambda,p_{N+1}}_{ P_{p_{N+1}}}$ is bounded by $K$.
By induction hypothesis, the multiplicity of singular spectrum for 
$$\left(\Hi^{\omega,\lambda,p_{N+1}}_{\sum_{i=1}^N P_{p_i}},A^{\omega,\lambda}_{p_{N+1}}\right)$$
is at most $K$.
Using Lemma \ref{lem::SingInc1}, there exists a full Lebesgue measure set $S^\omega$ such that
$$E^{A^{\omega,\lambda}_{p_{N+1}}}_{sing}(S^\omega)\Hi^{\omega,\lambda,p_{N+1}}_{\sum_{i=1}^N P_{p_i}}\subseteq E^{A^{\omega,\lambda}_{p_{N+1}}}_{sing}(S^\omega)\Hi^{\omega,\lambda,p_{N+1}}_{ P_{p_{N+1}}}.$$
From spectral averaging we have 
$$E^{A^{\omega,\lambda}_{p_{N+1}}}_{sing}(\RR\setminus S^\omega)\Hi^{\omega,\lambda,p_{N+1}}_{ P_{p_{N+1}}}=\{0\}$$
for almost all $\lambda$ (w.r.t Lebesgue measure). Now the decomposition
$$\Hi^{\omega,\lambda,p_{N+1}}_{\sum_{i=1}^N P_{p_i}}=E^{A^{\omega,\lambda}_{p_{N+1}}}(S^\omega)\Hi^{\omega,\lambda,p_{N+1}}_{\sum_{i=1}^N P_{p_i}}\oplus E^{A^{\omega,\lambda}_{p_{N+1}}}(\RR\setminus S^\omega)\Hi^{\omega,\lambda,p_{N+1}}_{\sum_{i=1}^N P_{p_i}},$$
 gives
\begin{align*}
 & E^{A^{\omega,\lambda}_{p_{N+1}}}_{sing}\Hi^{\omega,\lambda,p_{N+1}}_{\sum_{i=1}^{N+1} P_{p_i}}=E^{A^{\omega,\lambda}_{p_{N+1}}}_{sing}\Hi^{\omega,\lambda,p_{N+1}}_{\sum_{i=1}^{N} P_{p_i}}+E^{A^{\omega,\lambda}_{p_{N+1}}}_{sing}\Hi^{\omega,\lambda,p_{N+1}}_{ P_{p_{N+1}}}\\
 &\qquad\qquad=E^{A^{\omega,\lambda}_{p_{N+1}}}_{sing}(\RR\setminus S^\omega)\Hi^{\omega,\lambda,p_{N+1}}_{\sum_{i=1}^{N} P_{p_i}}\oplus E^{A^{\omega,\lambda}_{p_{N+1}}}_{sing}(S^\omega)\Hi^{\omega,\lambda,p_{N+1}}_{ P_{p_{N+1}}},
\end{align*}
where both the subspaces have multiplicity at most $K$. 
The supports of the singular spectrum of $A^{\omega,\lambda}_{p_{N+1}}$ restricted over the two subspaces are disjoint and this proves the induction hypothesis.
So this completes the first part of the proof.

With the induction completed, note that 
$$\Hi^\omega_{\sum_{i=1}^{N}P_{p_i}}\subseteq \Hi^\omega_{\sum_{i=1}^{N+1}P_{p_i}}\qquad\forall N\in\NN,$$
which implies $\tilde{\Hi}^\omega:=\cup_{n\in\NN}\Hi^\omega_{\sum_{i=1}^{N}P_{p_i}}$ is a linear subspace of $\Hi$, and it is dense because $\sum_{p\in\Nc}P_p=I$.
Clearly the space $\tilde{\Hi}^\omega$ is invariant under the action of $A^\omega$. 
For any finite collection $\{\phi_i\}_{i=1}^N\in\tilde{\Hi}^\omega$, there exists $M\in\NN$ such that $\phi_i\in\Hi^\omega_{\sum_{j=1}^M P_{p_j}}$ for all $i$.
So the multiplicity of the singular spectrum for $\tilde{\Hi}^\omega$ is bounded by $K$.
Hence using the density of $\tilde{\Hi}^\omega$ in $\Hi$, we get that the multiplicity of the singular spectrum is  bounded by $K$.

\end{proof}
\section{Application}\label{sec5}
For proving the Corollary \ref{corMultCouEx1} or Theorem \ref{thmSimBetheOp}, we need to obtain results about the multiplicity of the matrix $\sqrt{C_n}G^\omega_{n,n}(z)\sqrt{C_n}$. 
This is done by using resolvent equation for a special decomposition of $A^\omega$. 

Let $n\in\Nc$ be fixed, then using the fact that $range(C_n)\subset \mathcal{D}(A)$ the operators $P_nAP_n$, $(I-P_n)AP_n$ and $P_nA(I-P_n)$ are well defined, 
and since they are finite rank operators, they are bounded. Hence using the resolvent equation between $A^\omega$ and
$$\tilde{A}^\omega=P_n AP_n+(I-P_n)A(I-P_n)+\sum_{m\in\Nc}\omega_m C_m,$$
we obtain
\begin{align}\label{eqn::GreenFuncAppEq1}
G^\omega_{n,n}(z)=\left[ P_nAP_n+\omega_nC_n-z P_n-P_nA(I-P_n)(\tilde{A}^\omega-z)^{-1}(I-P_n)AP_n\right]^{-1},
\end{align}
where the operator on RHS is viewed as a linear operator on $P_n\Hi$.

So the maximum algebraic multiplicity of eigenvalues of $\sqrt{C_n}G^\omega_{n,n}(z)\sqrt{C_n}$ is same as the maximum algebraic multiplicity of eigenvalues of
\begin{equation}\label{eqn::GreenFuncAppEq2}
C_n^{-\frac{1}{2}}AC_n^{-\frac{1}{2}}-zC_n^{-1}-C_n^{-\frac{1}{2}} A(I-P_n)(\tilde{A}^\omega-z)^{-1}(I-P_n)A C_n^{-\frac{1}{2}}.
\end{equation}
Notice that above equation is independent of $\omega_n$.
The basic difference between the proof of Corollary \ref{corMultCouEx1} and Theorem \ref{thmSimBetheOp} is how the term
$$C_n^{-\frac{1}{2}} A(I-P_n)(\tilde{A}^\omega-z)^{-1}(I-P_n)A C_n^{-\frac{1}{2}}$$
is handled. Since the norm of above operator is $O((\Im z)^{-1})$, it is clear that we can ignore this term by choosing $\Im z$ large enough, 
but this term provides the simplicity of the spectrum in Theorem \ref{thmSimBetheOp}.

We will be using the following lemma:
\begin{lemma}\label{lem::MultBouInd}
Consider the operator $A^\omega$ and $A$ satisfying the hypothesis of corollary \ref{corMultCouEx1}. 
Let $I$ be a bounded interval contained in $(-\infty,M)$ such that maximum algebraic multiplicity of eigenvalues of $\sqrt{C_n}G^\omega_{n,n}(E)\sqrt{C_n}$ is bounded by $K$, for $E\in I$. 
Then for almost all $z$ the maximum algebraic multiplicity of $\sqrt{C_n}G^\omega_{n,n}(z)\sqrt{C_n}$ is bounded by $K$.
\end{lemma}
\begin{remark}
The main advantage of this lemma is that instead of looking for a bound in $\CC\setminus(M,\infty)$, we can work with $z\in\RR\setminus (\sigma(A^\omega)\cup \sigma(A))$ and so the operator $P_n(A^\omega-E)^{-1}P_n=\lim_{\epsilon\downarrow 0} P_n(A^\omega-E-\iota \epsilon)^{-1}P_n$ is self adjoint, hence the algebraic and geometric multiplicities coincides. 
\end{remark}
The proof follows same steps as the proof of Lemma \ref{lem::MultBou} and so we are omitting it here.
Now we are ready to prove the other two results.
\subsection{Proof of Corollary \ref{corMultCouEx1}}
Using Lemma \ref{lem::MultBouInd} and the fact that the algebraic multiplicity of $\sqrt{C_n} G^\omega_{n,n}(E) \sqrt{C_n}$ is same as algebraic multiplicity  of 
\begin{equation}\label{eqn::GreenFuncAppEq3}
C_n^{-\frac{1}{2}}AC_n^{-\frac{1}{2}}-EC_n^{-1}-C_n^{-\frac{1}{2}} A(I-P_n)\left(\tilde{A}^\omega-E\right)^{-1}(I-P_n)A C_n^{-\frac{1}{2}},
\end{equation} 
bounding the multiplicity of above equation for $E\ll M$ is enough. 

First we handle the case when $C_n$ are projections. The maximum algebraic multiplicity of \eqref{eqn::GreenFuncAppEq3} is same as
\begin{equation}\label{eqn::GreenFuncAppEq4}
P_nAP_n-P_n A(I-P_n)\left(\tilde{A}^\omega-E\right)^{-1}(I-P_n)A P_n,
\end{equation}
we can ignore the $EC_n^{-1}$ term because it is the identity operator, and so does not affect the multiplicity. 
Let
$$\delta=\min_{\substack{x,y\in\sigma(P_nAP_n)\\ x\neq y}} |x-y|,$$
then for $E<-M-\frac{3}{\delta}\norm{P_nA(I-P_n)}^2$ we have
$$\norm{P_n A(I-P_n)\left(\tilde{A}^\omega-E\right)^{-1}(I-P_n)A P_n}<\frac{\delta}{3}.$$
So viewing $P_n A(I-P_n)\left(\tilde{A}^\omega-E\right)^{-1}(I-P_n)A P_n$ as a perturbation, 
we get that any eigenvalue of \eqref{eqn::GreenFuncAppEq4} is in $\frac{\delta}{3}$ neighborhood of eigenvalues of $P_nAP_n$. 
So the multiplicity of any eigenvalue of \eqref{eqn::GreenFuncAppEq4} cannot exceed the multiplicity of the eigenvalues of $P_nAP_n$. 
This completes the proof for the case of projection.

For general $C_n$, the maximum algebraic multiplicity of \eqref{eqn::GreenFuncAppEq3} is same as the maximum algebraic multiplicity of 
\begin{equation}\label{eqn::CorMultEq1}
 -C_n^{-1}+\frac{1}{E}\left(C_n^{-\frac{1}{2}}AC_n^{-\frac{1}{2}}-C_n^{-\frac{1}{2}} A(I-P_n)\left(\tilde{A}^\omega-E\right)^{-1}(I-P_n)A C_n^{-\frac{1}{2}}\right),
\end{equation}
so setting
$$\delta=\min_{\substack{x,y\in\sigma(C_n^{-1})\\ x\neq y}} |x-y|$$
and choosing 
$$E<-2M-\frac{3}{\delta}\left(\norm{C_n^{-\frac{1}{2}}AC_n^{-\frac{1}{2}}}+\norm{C_n^{-\frac{1}{2}} A(I-P_n)}^2\right),$$
we get that the eigenvalues of \eqref{eqn::CorMultEq1} are in $\frac{\delta}{3}$ neighborhood of $C_n^{-1}$.
So following the argument for projection case we get that the multiplicity of any eigenvalue of \eqref{eqn::GreenFuncAppEq3} is upper bounded by the multiplicity of the eigenvalues of $C_n^{-1}$.

\subsection{Proof of Theorem \ref{thmSimBetheOp}}
Since $P_n\Delta_{\mathcal{B}}P_n$ has a non-trivial multiplicity, previous argument does not give us the desired result. So we have to concentrate on \eqref{eqn::GreenFuncAppEq4}, which in this case is
\begin{equation}\label{eqn::BetheOpEq1}
P_n\Delta_{\mathcal{B}}P_n-P_n \Delta_{\mathcal{B}}(I-P_n)\left(\tilde{H}^\omega-E\right)^{-1}(I-P_n)\Delta_{\mathcal{B}} P_n, 
\end{equation}
where
$$\tilde{H}^\omega=P_n\Delta_{\mathcal{B}}P_n+(I-P_n)\Delta_{\mathcal{B}}(I-P_n)+\sum_{x\in J}\omega_x P_x.$$
Here we denote $P_x=\chi_{\tilde{\Lambda}(x)}$. For simplicity of notation let us denote
$$\partial \tilde{\Lambda}(x)=\{(p,q)\in \tilde{\Lambda}(x)\times \tilde{\Lambda}(x)^c: d(p,q)=1\},$$
i.e we pair all the leaf nodes of the tree $\tilde{\Lambda}(x)$ with its neighbors outside the tree. 

\begin{figure}[ht]
 \begin{center}
  \begin{tikzpicture}[scale=2.5]

\filldraw[gray!30] (0,0.06)--(0.95,-0.5)--(1.48,-1.1)--(-1.48,-1.1)--(-0.95,-0.5)--cycle;
\draw(0,-0.3)node[below]{\small $\tilde{\Lambda}(x)$};

\filldraw[gray](0,1)--(-0.4,0.2)--(0.4,0.2)--cycle;
\draw[thick](0,0.2)--(0,0);
\draw(0,0.5)node[below]{\small $\mathcal{T}_0$};

\coordinate (r) at (0,0);
\filldraw(r)circle[radius=0.05]node[right]{\small$~~0_l:=x$};

\coordinate (n1) at (0.866,-0.5);
\coordinate (n2) at (-0.866,-0.5);
\filldraw(n1)circle[radius=0.05];
\filldraw(n2)circle[radius=0.05];

\draw[thick] (r)--(n1);
\draw[thick] (r)--(n2);

\coordinate (n11) at (0.45,-1.0);
\coordinate (n21) at (-0.45,-1.0);
\filldraw(n11)circle[radius=0.05];
\filldraw(n21)circle[radius=0.05];

\draw[thick] (n1)--(n11);
\draw[thick] (n2)--(n21);

\coordinate (n12) at (1.332,-1.0);
\coordinate (n22) at (-1.332,-1.0);
\filldraw(n12)circle[radius=0.05];
\filldraw(n22)circle[radius=0.05];

\draw[thick] (n1)--(n12);
\draw[thick] (n2)--(n22);

\coordinate(n22b2) at (-1.582,-1.25);
\draw[thick](n22)--(n22b2);
\filldraw[gray](-1.582,-1.25)--(-1.432,-1.75)--(-1.732,-1.75)--cycle;
\filldraw(n22b2)circle[radius=0.03];
\draw(-1.582,-1.45)node[below]{\tiny $\mathcal{T}_1$};

\coordinate(n22b1) at (-1.082,-1.25);
\draw[thick](n22)--(n22b1);
\filldraw[gray](-1.082,-1.25)--(-0.932,-1.75)--(-1.232,-1.75)--cycle;
\filldraw(n22b1)circle[radius=0.03];
\draw(-1.082,-1.45)node[below]{\tiny $\mathcal{T}_2$};

\coordinate(n21b2) at (-0.65,-1.25);
\draw[thick](n21)--(n21b2);
\filldraw[gray](-0.65,-1.25)--(-0.50,-1.75)--(-0.8,-1.75)--cycle;
\filldraw(n21b2)circle[radius=0.03];
\draw(-0.65,-1.45)node[below]{\tiny $\mathcal{T}_3$};

\coordinate(n21b1) at (-0.25,-1.25);
\draw[thick](n21)--(n21b1);
\filldraw[gray](-0.25,-1.25)--(-0.10,-1.75)--(-0.4,-1.75)--cycle;
\filldraw(n21b1)circle[radius=0.03];
\draw(-0.25,-1.45)node[below]{\tiny $\mathcal{T}_4$};

\coordinate(n12b2) at (1.582,-1.25);
\draw[thick](n12)--(n12b2);
\filldraw[gray](1.582,-1.25)--(1.432,-1.75)--(1.732,-1.75)--cycle;
\filldraw(n12b2)circle[radius=0.03];
\draw(1.582,-1.45)node[below]{\tiny $\mathcal{T}_8$};

\coordinate(n12b1) at (1.082,-1.25);
\draw[thick](n12)--(n12b1);
\filldraw[gray](1.082,-1.25)--(0.932,-1.75)--(1.232,-1.75)--cycle;
\filldraw(n12b1)circle[radius=0.03];
\draw(1.082,-1.45)node[below]{\tiny $\mathcal{T}_7$};

\coordinate(n11b2) at (0.65,-1.25);
\draw[thick](n11)--(n11b2);
\filldraw[gray](0.65,-1.25)--(0.50,-1.75)--(0.80,-1.75)--cycle;
\filldraw(n11b2)circle[radius=0.03];
\draw(0.65,-1.45)node[below]{\tiny $\mathcal{T}_6$};

\coordinate(n11b1) at (0.25,-1.25);
\draw[thick](n11)--(n11b1);
\filldraw[gray](0.25,-1.25)--(0.10,-1.75)--(0.4,-1.75)--cycle;
\filldraw(n11b1)circle[radius=0.03];
\draw(0.25,-1.45)node[below]{\tiny $\mathcal{T}_5$};

\end{tikzpicture}

\end{center}
\caption{A representation of the rooted tree with three neighbors. Observe that removing the sub-tree $\tilde{\Lambda}(x)$ divides the graphs into nine connected components.}\label{fig3}
\end{figure}
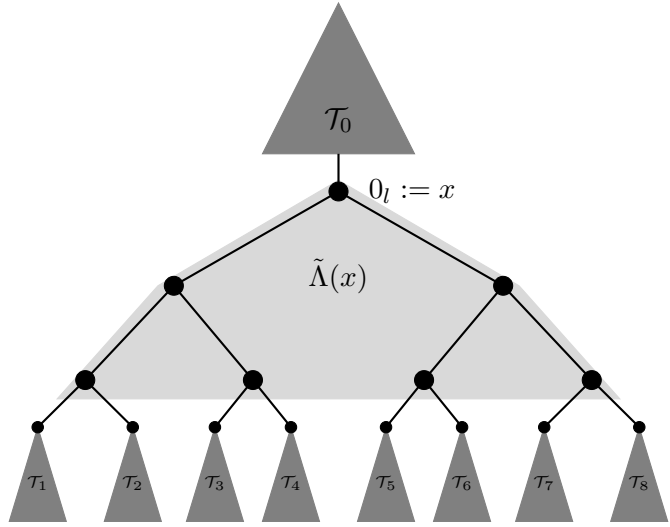

Following Dirac notation, observe that
\begin{align*}
&P_n \Delta_{\mathcal{B}}(I-P_n)\left(\tilde{H}^\omega-E\right)^{-1}(I-P_n)\Delta_{\mathcal{B}} P_n\\
&\qquad=\sum_{(p,q)\in \partial \tilde{\Lambda}(x)} \proj{\delta_p} \dprod{\delta_q}{(\tilde{H}^\omega-E)^{-1}\delta_q},
\end{align*}
this follows because 
$$\dprod{\delta_q}{(I-P_n)\Delta_{\mathcal{B}}P_n \delta_p}=\left\{\begin{matrix} 1 & (p,q)\in\partial\tilde{\Lambda}(n)\\ 0 & otherwise \end{matrix}\right.,$$
and 
$$\dprod{\delta_{q_1}}{(\tilde{H}^\omega)^k \delta_{q_2}}=0\qquad\forall k\in\NN$$ 
for $(p_1,q_1),(p_2,q_2)\in \partial\tilde{\Lambda}(n)$ and $q_1\neq q_2$.
This is also the reason why the random variables 
$$\left\{\dprod{\delta_q}{(\tilde{H}^\omega-E)^{-1}\delta_q}\right\}_{(p,q)\in\partial\tilde{\Lambda}(x)}$$ 
are independent of each other. 
The random variable $\dprod{\delta_q}{(\tilde{H}^\omega-E)^{-1}\delta_q}$ is real for $E\in\RR$, and has absolutely continuous distribution, which follows from the following expression
\begin{align*}
&\dprod{\delta_q}{(\tilde{H}^\omega-E)^{-1}\delta_q}\\
&\qquad\qquad=\cfrac{1}{\omega_q-E-\sum_{x_1\in N_{q}}\cfrac{1}{\omega_{x_1}-E-\sum_{x_2\in N_{x_1}} \cfrac{1}{\ddots-\sum_{x_l\in N_{x_{l-1}}} a^\omega_{x_l}(E)}}},
\end{align*}
where $\{a^\omega_{x_l}(E)\}$ are independent of $\omega_q$, and the distribution of $\omega_q$ is absolutely continuous with respect to the Lebesgue measure. 
Now Theorem \ref{thmSimBetheOp} follows from Theorem \ref{thmSimTreePert}. 

But first few notations are needed.
Denote $\mathcal{T}_L$ to be a rooted tree with root $0_L$ and every vertex have $K+1$ neighbors except root $0_L$ (which has $K$ neighbors) and vertices in the boundary
$$\partial\mathcal{T}_L:=\{x\in\mathcal{T}_L: d(0_L,x)=L\}$$
which have one neighbor each.
\begin{theorem}\label{thmSimTreePert}
 Let $\Delta_{\mathcal{T}_L}$ denote the adjacency matrix over $\mathcal{T}_L$ and set 
$$B_\tau=\sum_{x\in\partial \mathcal{T}_L}t_x\proj{\delta_x}$$
for $\tau=\{t_x\}_{x\in\partial \mathcal{T}_L}\in \RR^{\partial \mathcal{T}_L}$. Then for almost all $\tau$ w.r.t. the Lebesgue measure, the spectrum of $H_\tau=\Delta_{\mathcal{T}_L}+B_\tau$ is simple.
\end{theorem}
\begin{proof}
The proof is done by induction on $L$.
For the proof  denote $H_{\tau,l}$ to be the operator
$$H_{\tau,l}=\Delta_{\mathcal{T}_l}+\sum_{x\in\partial \mathcal{T}_l}\tau_x \proj{\delta_x}$$
where $\Delta_{\mathcal{T}_l}$ is the adjacency operator on the rooted tree $\mathcal{T}_l$ with root $0_l$.

The induction is done over the statement \emph{ For almost all $\tau$, $H_{\tau,l}$ has simple spectrum with the property that all the eigenfunctions are non-zero at root, 
and $\sigma(H_{\tau,l})\cap \sigma(H_{\omega,l})=\phi$ for almost all $\omega$}.

For $l=0$, the statement is trivial because $H_{\tau,0}$ is the operator on $\CC$ which is multiplication by the random variable $\tau_{0_l}$.

For the induction step suppose that the statement holds for all $l=N-1$. Observe that
$$H_{\tau,N}=\sum_{x:d(0_N,x)=1}(|\delta_{0_N}\rangle\langle \delta_x|+|\delta_x\rangle\langle \delta_{0_N}|)+\sum_{x:d(0_N,x)=1}H_{\tau,x},$$
where $H_{\tau,x}:=\chi_{\mathcal{T}_{x}} H_{\tau,l}\chi_{\mathcal{T}_{x}}$ for the sub-tree $\mathcal{T}_{x}:=\{y\in\mathcal{T}_l: d(0_N,y)=d(0_N,x)+d(x,y)\}$.

\begin{figure}[ht]
 \begin{center}
 \begin{tikzpicture}
\filldraw[gray](-2,1.5)--(-2.4,0.0)--(-1.6,0.0)--cycle;\draw(-2,0.5)node{\tiny$\mathcal{T}_{x_1}$};
\filldraw[gray](-1,1.5)--(-1.4,0.0)--(-0.6,0.0)--cycle;\draw(-1,0.5)node{\tiny$\mathcal{T}_{x_2}$};
 
\filldraw[gray](2,1.5)--(2.4,0.0)--(1.6,0.0)--cycle;\draw(2,0.5)node{\tiny$\mathcal{T}_{x_K}$};
\filldraw[gray](1,1.5)--(1.4,0.0)--(0.6,0.0)--cycle;\draw(1,0.5)node{\tiny$\mathcal{T}_{x_{K-1}}$};

\filldraw(0,2)circle(0.5pt);\draw(0,2)node[above]{\tiny$0_l$};
\filldraw(-2,1.5)circle(0.5pt);\draw(-2,1.5)node[left]{\tiny$x_1$};
\filldraw(2,1.5)circle(0.5pt);\draw(2,1.5)node[right]{\tiny$x_K$};
\filldraw(-1,1.5)circle(0.5pt);\draw(-1,1.5)node[left]{\tiny$x_2$};
\filldraw(1,1.5)circle(0.5pt);\draw(1,1.5)node[right]{\tiny$x_{K-1}$};
\draw(0,2)--(-2,1.5);
\draw(0,2)--(-1,1.5);
\draw(0,2)--(2,1.5);
\draw(0,2)--(1,1.5);

\filldraw(0,1)circle(0.3pt);
\filldraw(-0.1,1)circle(0.3pt);
\filldraw(0.1,1)circle(0.3pt);
\filldraw(-0.2,1)circle(0.3pt);
\filldraw(0.2,1)circle(0.3pt);
 
\end{tikzpicture}
 \end{center}
\caption{The tree $\mathcal{T}_l$ can be viewed as a union of $K$ disjoint trees $\{\mathcal{T}_{x_i}\}_i$ which are connected through their roots $\{x_1,\cdots,x_K\}$ to a separate node $0_l$.}\label{fig4}
\end{figure}
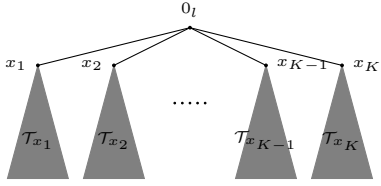

First notice that $H_{\tau,x}$ is unitarily equivalent to $H_{\tilde{\tau},N-1}$ where $\tilde{\tau}$ is restriction of $\tau$ onto the $\partial\mathcal{T}_x$.
Next note that $\{\tau_y\}_y$ that appear in $H_{\tau,x_i}$ are disjoint for two subtrees $\mathcal{T}_{x_1}$ and $\mathcal{T}_{x_2}$ for $x_1\neq x_2$.
Hence by induction hypothesis we have $\sigma(H_{\tau,x})\cap \sigma(H_{\tau,y})=\phi$ for $x\neq y$ and the spectrum of $H_{\tau,x}$ is simple with the property that the eigenfunctions corresponding to the eigenvalues are non-zero at the root, for each $x$.

Since we are working on tree graphs, we have
\begin{align}
\dprod{\delta_{0_N}}{(H_{\tau,N}-z)^{-1}\delta_{0_N}}&=\frac{1}{-z-\sum_{x:d(0_N,x)=1}\dprod{\delta_{x}}{(H_{\tau,x}-z)^{-1}\delta_{x}}}\nonumber\\
&=\frac{1}{-z-\sum_{x:d(0_N,x)=1}\sum_{E\in\sigma(H_{\tau,x})}\frac{\left|\dprod{\psi_{\tau,x,E}}{\delta_x}\right|^2}{E-z}}\label{eqn::treeEq1}
\end{align}
where $\psi_{\tau,x,E}$ is the eigenfunction of $H_{\tau,x}$ for the eigenvalue $E$. 
By the induction hypothesis we have $\dprod{\psi_{\tau,x,E}}{\delta_x}\neq 0$ for each $E\in\sigma(H_{\tau,x})$ and $x$ a neighbor of $0_N$ .
Next using the fact that $\sigma(H_{\tau,x})\cap \sigma(H_{\tau,y})=\phi$ for $x\neq y$, we get that
$$z+\sum_{x:d(0_N,x)=1}\sum_{E\in\sigma(H_{\tau,x})}\frac{\left|\dprod{\psi_{\tau,x,E}}{\delta_x}\right|^2}{E-z}$$
has $\sum_{x:d(0_N,x)=1}\#\sigma(H_{\tau,x})$ many poles and so the equation \eqref{eqn::treeEq1} has 
$$1+\sum_{x:d(0_N,x)=1}\#\sigma(H_{\tau,x})$$
many roots, which is equal to $|\mathcal{T}_N|$. 
But using functional calculus we also have 
\begin{equation*}
\dprod{\delta_{0_N}}{(H_{\tau,N}-z)^{-1}\delta_{0_N}}=\sum_{E\in\sigma(H_{\tau,N})}\frac{|\dprod{\psi_{\tau,N,E}}{\delta_{0_N}}|^2}{E-z}
\end{equation*}
where $\psi_{\tau,N,E}$ is the eigenfunction corresponding to the eigenvalue $E$ for the matrix $H_{\tau,N}$.
So each pole $\dprod{\delta_{0_N}}{(H_{\tau,N}-z)^{-1}\delta_{0_N}}$ corresponds to an eigenvalue, and previous computation shows that there are $|\mathcal{T}_N|$ many poles, which gives the simplicity of the spectrum of $H_{\tau,N}$. 
Finally, the eigenfunction $\psi_{\tau,N,E}$ is non-zero at the root $0_N$ because of the fact that if $\dprod{\psi_{\tau,N,E}}{\delta_{0_N}}=0$, then the pole corresponding to $E$ will not be present in the above expression.

Finally we have to prove  $\sigma(H_{\tau,l})\cap \sigma(H_{\omega,l})=\phi$ for almost all $\tau,\omega$. But first we need the following claim:
\\
\noindent{\bf Claim:} For any solution $\psi\in\CC^{\mathcal{T}_l}\setminus\{0\}$ of $H_{\tau,l}\psi=E\psi$ for $E\in\RR$, there exists $x\in\partial\mathcal{T}_l$ such that $\psi_x\neq 0$.\\
{\it proof:} If for some $E\in\RR$ there exists $\psi\in\CC^{\mathcal{T}_l}$ such that $H_{\tau,l}\psi=E\psi$ and 
$$\psi_x=0\qquad\forall x\in\partial\mathcal{T}_l,$$
then for any $x\in\partial\mathcal{T}_l$
\begin{align*}
 & (H_{\tau,l}\psi)_x=E\psi_x=0\\
 \Rightarrow\qquad& \psi_{Px}+t_x\psi_x=0\\
 \Rightarrow\qquad& \psi_{Px}=0,
\end{align*}
where $Px$ is the unique neighbor of $x$ satisfying $d(0_l,x)=d(0_l,Px)+1$. So we get that $\psi_x=0$ for all $x\in\mathcal{T}_l$ such that $d(0,x)=l-1$. 
Repeating the above argument for $x$ satisfying $d(0,x)=l-1$ will give $\psi_x=0$ for all $x$ such that $d(0_l,x)=l-2$.  
Repeating the last step recursively gives $\psi\equiv0$ giving contradiction, which completes the proof of the claim.
\\
\\
Now to prove $\sigma(H_{\tau,l})\cap \sigma(H_{\omega,l})=\phi$ for almost all $\tau,\omega$.
Denote $\tau=\{\tau_x\}_{x\in\partial \mathcal{T}_l}$, $\omega=\{\omega_x\}_{x\in\partial \mathcal{T}_l}$, 
set $\{E^\tau_{i}\}_i$ and $\{\psi^\tau_i\}$ to be the eigenvalues and the corresponding eigenfunctions for $H_{\tau,l}$ and similarly for $H_{\omega,l}$.
Using Feynman-Hellmann theorem for rank one perturbation, we have
\begin{equation*}
 \frac{dE^\tau_i}{d\tau_x}=|\dprod{\psi^\tau_{i}}{\delta_x}|^2\qquad\forall x\in\partial\mathcal{T}_l,\forall i,
\end{equation*}
and similarly 
\begin{equation*}
 \frac{dE^\omega_i}{d\omega_x}=|\dprod{\psi^\omega_{i}}{\delta_x}|^2\qquad\forall x\in\partial\mathcal{T}_l,\forall i.
\end{equation*}
For each $i$, using the previous claim, there exists $x^\tau_i\in\partial \mathcal{T}_l$ such that $\dprod{\psi^\tau_{i}}{\delta_{x^\tau_i}}\neq 0$, and similarly for $\omega$.
Now using Implicit Function Theorem over $E^\tau_{i}-E^\omega_j=0$, the manifold
$$\{(\tau,\omega)\in\RR^{\partial\mathcal{T}_l}\times\RR^{\partial\mathcal{T}_l} : E^\tau_{i}=E^\omega_j\}$$
has lower dimension than $2|\partial\mathcal{T}_l|$. So in particular 
$$Leb\left(\{(\tau,\omega)\in\RR^{\partial\mathcal{T}_l}\times\RR^{\partial\mathcal{T}_l} : E^\tau_{i}=E^\omega_j\}\right)=0$$
which completes the proof of the induction step.

\end{proof}

\noindent{\bf Acknowledgement:} The author, Dhriti Ranjan Dolai is supported by from the J. C. Bose Fellowship grant of Prof. B. V. Rajarama Bhat.
\appendix
\section{Appendix}
\begin{lemma}\label{lem::ProjSub}
On a separable Hilbert space $\Hi$, let $H$ be a self adjoint operator, 
and for $\phi,\psi\in\Hi$ set $\sigma_\phi(\cdot)=\dprod{\phi}{E_H(\cdot)\phi}$ and $\sigma_{\phi,\psi}(\cdot)=\dprod{\phi}{E_H(\cdot)\psi}$.
Let $f$ be the Radon-Nikodym derivative of $\sigma_{\phi,\psi}$ w.r.t $\sigma_\phi$, then $f(H)\phi$ is the projection of $\psi$ onto the minimal closed $H$-invariant subspace containing $\phi$.
\end{lemma}
\begin{proof}
Let $\Hi_\phi$ denote the minimal closed $H$-invariant subspace containing $\phi$, 
then $(\Hi_\phi,H)$ is unitarily equivalent to $(L^2(\RR,\sigma_\phi),M_{Id})$ where $M_{Id}$ is multiplication with the identity map on $\RR$. 
We have the linear functional 
$$g\mapsto \dprod{g(H)\phi}{\psi-f(H)\phi}$$
for $g\in L^2(\RR,\sigma_\phi)$. Observe that 
\begin{align*}
\dprod{g(H)\phi}{\psi-f(H)\phi}=\dprod{g(H)\phi}{\psi}-\dprod{g(H)}{f(H)\phi}\\
=\int g(x)d\sigma_{\phi,\psi}(x)-\int g(x)f(x)d\sigma_{\phi}(x)=0,
\end{align*}
because $f$ is Radon-Nikodym derivative of $\sigma_{\phi,\psi}$ with respect to $\sigma_\phi$. Since $g(H)\phi$ are dense in $\Hi_\phi$ for $\phi\in L^2(\RR,\sigma_\phi)$, we have
$$\psi-f(H)\phi\perp \Hi_\phi,$$
hence $f(H)\phi$ is the projection of $\psi$ on to $\Hi_\phi$.

\end{proof}

\begin{lemma}\label{lem::SumHilSub}
On a separable Hilbert space $\Hi$ let $H$ be a self-adjoint operator and $Q$ be a finite ranked projection. Let $\{e_i\}_{i\in\NN}$ be a orthonormal basis for the subspace $Q\Hi$ and denote
$$\Hi_i=\clsp{f(H)e_i:~f\in C_c(\RR)},$$
and
$$\Hi_Q=\clsp{f(H)\phi:~f\in C_c(\RR)~\&~\phi\in Q\Hi}.$$
Then
$$\Hi_Q=\sum_i \Hi_i,$$
where $\sum_i \Hi_i$ denotes the closure of finite linear combinations of elements of $\Hi_i$.
\end{lemma}
\begin{proof}
Since $\Hi_i\subseteq \Hi_Q$ for any $i$, we always have
$$\sum_i \Hi_i\subseteq \Hi_Q.$$
For the other way round note that we only have to show $f(H)\phi\in\sum_i \Hi_i$, for $\phi\in Q\Hi$. Since $\{e_i\}_i$ is a basis, we have
$$\phi=\sum_i a_ie_i.$$
Using it, define
$$\psi_N=\sum_{i=1}^N  a_i f(H)e_i,$$
which satisfies $\psi_N\in\sum_i\Hi_i$ for any $N\in\NN$. Now the conclusion of the lemma holds, since $\sum_i\Hi_i$ is closed.
\end{proof}

\bibliographystyle{plain}

\end{document}